\documentclass[12pt,onecolumn,journal,draftclsnofoot,doublespace]{IEEEtran}
\usepackage{times,amsmath,amssymb,graphicx,epsf,psfrag,color,epsfig,cite}

\def\boxit#1{\vbox{\hrule\hbox{\vrule\kern3pt
        \vbox{\kern3pt#1\kern3pt}\kern3pt\vrule}\hrule}}

\def\reals{ { {\rm  I \kern-0.15em R }  } }
\def\complex{ {\,{{\rm C} \kern-0.50em \raise0.20ex {  |}}\, }}

\def\Rbf{{\bf R}}

\def\be{\begin{equation}}
\def\ee{\end{equation}}

\def\defeq{{\stackrel{\Delta}{=}}}
\def\scalefig#1{\epsfxsize #1\textwidth}
\def\Rxx{\Rbf_{\ssstyle X\kern-.1em X}}

\let\ssstyle=\scriptscriptstyle

\def\eg{{\it e.g.,\ \/}}
\def\etal{{\it et al. \/}}

\def\ie{{\it i.e.,\ \/}}
\def\Kout{\setbox1=\hbox{\Huge\bf K}\hbox to
1.05\wd1{\hspace{.05\wd1}
}}
\def\Sout{\setbox1=\hbox{\Huge\bf S}\hbox to 1.05\wd1{\hspace{.05\wd1}
}}

\IEEEoverridecommandlockouts

\def\ie{{\it i.e.,\ \/}}

\def\defeq{{\stackrel{\Delta}{=}}}

\def\scalefig#1{\epsfxsize #1\textwidth}

\newcommand{\mbbE}{\mathbb{E}}

\newtheorem{theorem}{Theorem}

\title{\Large\bf  Learning in A Changing World: Restless Multi-Armed Bandit with Unknown Dynamics}

\author{Haoyang Liu,  ~~~  Keqin Liu, ~~~ Qing Zhao\\
 University of California, Davis, CA 95616\\
 \{liu, kqliu, qzhao\}@ucdavis.edu}

\begin{document}
\date{}
\markboth{}{Liu and Liu and Zhao } \maketitle

\begin{abstract}
We consider the restless multi-armed bandit (RMAB) problem with
unknown dynamics in which a player chooses $M$ out of $N$ arms to
play at each time. The reward state of each arm transits according
to an unknown Markovian rule when it is played and evolves according
to an arbitrary unknown random process when it is passive. The
performance of an arm selection policy is measured by regret,
defined as the reward loss with respect to the case where the player
knows which $M$ arms are the most rewarding and always plays the $M$
best arms. We construct a policy with an interleaving exploration
and exploitation epoch structure that achieves a regret with
logarithmic order when arbitrary (but nontrivial) bounds on certain
system parameters are known. When no knowledge about the system is
available, we show that the proposed policy achieves a regret
arbitrarily close to the logarithmic order. We further extend the
problem to a decentralized setting where multiple distributed
players share the arms without information exchange. Under both an
exogenous restless model and an endogenous restless model, we show
that a decentralized extension of the proposed policy preserves the
logarithmic regret order as in the centralized setting.  The results
apply to adaptive learning in various dynamic systems and
communication networks, as well as financial investment.
\end{abstract}

\section{Introduction}\label{sec:introduction}

\subsection{Multi-Armed Bandit with i.i.d. and Rested Markovian Reward Models}
\footnotetext{This work was supported by the Army Research Office
under Grant W911NF-08-1-0467 and by the Army Research Lab under
Grant W911NF-09-D-0006. Part of these results were presented at ITA,
San Diego, CA, January 2011 and IEEE ICASSP, Prague, Czeck Republic,
May 2011.} In the classic multi-armed bandit (MAB) with an i.i.d.
reward model, there are $N$ independent arms and a single player.
Each arm, when played, offers an i.i.d. random reward drawn from a
distribution with unknown mean. At each time, the player chooses one
arm to play, aiming to maximize the total expected reward in the
long run. This problem involves the well-known tradeoff between
exploitation and exploration, where the player faces the conflicting
objectives of playing the arm with the best reward history and
playing a less explored arm to learn its reward statistics.

A commonly used performance measure of an arm selection policy is
the so-called \emph{regret} or \emph{the cost of learning,} defined
as the reward loss with respect to the case with a known reward
model. It is clear that under a known reward model, the player
should always play the arm with the largest reward mean. The essence
of the problem is thus to identify the best arm without engaging
other arms too often. Any policy with a sublinear growth rate of
regret achieves the same maximum average reward (given by the
largest reward mean) as in the known model case. However, the slower
the regret growth rate, the faster the convergence to this maximum
average reward, indicating a more effective learning ability of the
policy.

In~1985, Lai and Robbins showed that regret grows at least at a
logarithmic order with time, and an optimal policy was explicitly
constructed to achieve the minimum regret growth rate for several
reward distributions including Bernoulli, Poisson, Gaussian,
Laplace~\cite{Lai&Robbins85AAM}. Several other policies have been
developed under different assumptions on the reward
distribution~\cite{Agrawal:95,Auer:02}. In particular, an index
policy, referred to as Upper Confidence Bound 1 (UCB1) proposed by
Auer \etal in~\cite{Auer:02}, achieves logarithmic regret for any
reward distributions with finite support. In~\cite{qingzhaosub}, Liu
and Zhao proposed a policy that achieves the optimal logarithmic
regret order for a more general class of reward distributions and
sublinear regret orders for heavy-tailed reward distributions.

In 1987, Anantharam \etal extended Lai and Robbin's results to a
Markovian reward model where the reward state of each arm evolves as
an unknown Markov process over successive plays and remains frozen
when the arm is passive (the so-called \emph{rested} Markovian
reward model)~\cite{Anantharam:87-2}. In~\cite{Tekin:10}, Tekin and
Liu extended the UCB1 policy proposed in~\cite{Auer:02} to the
rested Markovian reward model.

\subsection{Restless Multi-Armed Bandit with Unknown Dynamics}

In this paper, we consider Restless Multi-Armed Bandit (RMAB), a
generalization of the classic MAB. In contrast to the rested
Markovian reward model, in RMAB, the state of each arm continues to
evolve even when it is not played. More specifically, the state of
each arm changes according to an unknown Markovian transition rule
when the arm is played and according to an arbitrary unknown random
process when the arm is not played. We consider both the centralized
(or equivalently, the single-player) setting and the decentralized
setting with multiple distributed players.

\subsubsection{Centralized Setting}

A centralized setting where $M$ players share their observations and
make arm selections jointly is equivalent to a single player who
chooses and plays $M$ arms simultaneously. The performance measure
regret is similarly defined: it is the reward loss compared to the
case when the player knows which arms are the most rewarding and
always plays the $M$ best arms.

Compared to the i.i.d. and the rested Markovian reward models, the
restless nature of arm state evolution requires that each arm be
played consecutively for a period of time in order to learn its
Markovian reward statistics. The length of each segment of
consecutive plays needs to be carefully controlled to avoid spending
too much time on a bad arm. At the same time, we experience a
transient  each time we switch out and then back to an arm, which
leads to potential reward loss compared to the steady-state behavior
of this arm. Thus, the frequency of arm switching needs to be
carefully bounded.

To balance these factors, we construct a policy based on a
deterministic sequencing of exploration and exploitation (DSEE) with
an epoch structure. Specifically, the proposed policy partitions the
time horizon into interleaving exploration and exploitation epochs
with carefully controlled epoch lengths. During an exploration
epoch, the player partitions the epoch into $N$ contiguous segments,
one for playing each of the $N$ arms to learn their reward
statistics. During an exploitation epoch, the player plays the arm
with the largest sample mean (\ie average reward per play)
calculated from the observations obtained so far. The lengths of
both the exploration and the exploitation epochs grow geometrically.
The number of arm switchings are thus at the logrithmic order with
time. The tradeoff between exploration and exploitation is balanced
by choosing the cardinality of the sequence of exploration epochs.
Specifically, we show that with an $O(\log t)$ cardinality of the
exploration epochs, sufficiently accurate learning of the arm ranks
can be achieved when arbitrary (but nontrivial) bounds on certain
system parameters are known, and the DSEE policy offers a
logarithmic regret order. When no knowledge about the system is
available, we can increase the cardinality of the exploration epochs
by an arbitrarily small order and achieve a regret arbitrarily close
to the logarithmic order, \ie the regret has order $f(t)\log t$ for
any increasing divergent function $f(t)$. In both cases, the
proposed policy achieves the maximum average reward offered by the
$M$ best arm.

We point out that the definition of regret here, while similar to
that used for the classic MAB, is a weaker version of its
counterpart. In the classic MAB with either i.i.d. or rested
Markovian rewards, the optimal policy under a known model is indeed
to stay with the best arm in terms of the reward mean\footnote{Under
the rested Markovian reward model, staying with the best arm (in
terms of the steady-state reward mean) is optimal up to a loss of
$O(1)$ term resulting from the transient effect of the initial state
which may not be the stationary distribution~\cite{Anantharam:87-2}.
This $O(1)$ term , however, does not affect the order of the
regret.}. For RMAB, however, the optimal policy under a known model
is no longer given by staying with the arm with the largest reward
mean. Unfortunately, even under known Markovian dynamics, RMAB has
been shown to be P-SPACE hard~\cite{Papadimitriou:99}. In this
paper, we adopt a weaker definition of regret. First introduced in
\cite{Auer-nonsto}, weak regret measures the performance of a policy
against a ``partially-informed'' genie who knows only which arm has
the largest reward mean instead of the complete system dynamics.
This definition of regret leads to a tractable problem, but at the
same time, weaker results. Whether stronger results for a general
RMAB under an unknown model can be obtained is still open for
exploration (see more discussions in Sec.~\ref{sec:relatedwork} on
related work).

\subsubsection{Decentralized Setting}

In the decentralized setting, there are $M$ distributed players. At
each time, a player chooses one arm to play based on its local
observations without information exchange with other players.
Collisions occur when multiple players choose the same arm and
result in reward loss. The objective here is a decentralized policy
to minimize the regret growth rate where regret is defined as the
performance loss with respect to the ideal case where the players
know the $M$ best arms and are perfectly orthogonalized among these
$M$ best arms through centralized scheduling.

 We consider two types of restless reward
models: the exogenous restless model and the endogenous restless
model. In the former, the system itself is rested: the state of an
arm does not change when the arm is not engaged. However, from each
individual player's perspective, arms are restless due to actions of
other players that are unobservable and uncontrollable. Under the
endogenous restless model, the state of an arm evolves according to
an arbitrary unknown random process even when the arm is not played.
Under both restless models, we extend the proposed DSEE policy to a
decentralized policy that achieves the same logarithmic regret order
as in the centralized scheduling. We emphasize that the logarithmic
regret order is achieved under a complete decentralization among
players. Players do not need to have synchronized global timing;
each player can construct the exploration and exploitation epoch
sequences according to its own local time.

We point out that the result under the exogenous restless model is
stronger than that under the endogenous restless model in the sense
that the regret is indeed defined with respect to the optimal policy
under a known reward model and centralized scheduling. This is
possible due to the inherent \emph{rested} nature of the systems
which makes any orthogonal sharing of the $M$ best arms optimal (up
to an $O(1)$ term) under a known reward model.

\subsection{Related Work on RMAB}\label{sec:relatedwork}
RMAB with unknown dynamics has not been studied in the literature
except two parallel independent investigations reported
in~\cite{Tekin:10-2} and \cite{QZ10}, both consider only a single
player. In~\cite{Tekin:10-2}, Tekin and Liu considered the same
problem and adopted the same definition of regret as in this paper.
They proposed a policy that achieves logarithmic (weak) regret when
certain knowledge about the system parameters is
available~\cite{Tekin:10-2}. Referred to as regenerative cycle
algorithm (RCA), the policy proposed in~\cite{Tekin:10-2} is based
on the UCB1 policy proposed in~\cite{Auer:02} for the i.i.d. reward
model. The basic idea of RCA is to play an arm consecutively for a
random number of times determined by a regenerative cycle of a
particular state and arms are selected based on the UCB1 index
calculated from observations obtained only inside the regenerative
cycles (observations obtained outside the regenerative cycles are
not used in learning). The i.i.d. nature of the regenerative cycles
reduces the problem to the classic MAB under the i.i.d. reward
model. The DSEE policy proposed in this paper, however, has a
deterministic epoch structure, and all observations are used in
learning. As shown in the simulation examples in
Sec.~\ref{sec:simulation}, DSEE can offer better performance than
RCA since RCA may have to discard a large number of observations
from learning before the chosen arm enters a regenerative cycle
defined by a particular pilot state. Note that when the arm reward
state space is large or when the chose pilot state has a small
stationary probability, it may take a long time for the arm to hit
the pilot state, and since the transition probabilities are unknown,
it is difficult to choose the pilot state for a smaller hitting
time. In~\cite{QZ10}, a strict definition of regret was adopted (\ie
the reward loss with respect to the optimal performance in the ideal
scenario with a known reward model). However, the problem can only
be solved for a special class of RMAB with $2$ or $3$ arms governed
by stochastically identical two-state Markov chains. For this
special RMAB, the problem is tractable due to the semi-universal
structure of the optimal policy of the corresponding RMAB with known
dynamics established in~\cite{QZTWC,CSZ}. By exploiting the simple
structure of the optimal policy under known Markovian dynamics, Dai
\etal showed in ~\cite{QZ10} that a regret with an order arbitrarily
close to logarithmic can be achieved for this special RMAB.

There are also several recent development on decentralized MAB with
multiple players under the i.i.d. reward model. In~\cite{Keqin:10},
Liu and Zhao proposed a Time Division Fair Sharing (TDFS) framework
which leads to a family of decentralized fair policies that achieve
logarithmic regret order under general reward distributions and
observation models~\cite{Keqin:10}. Under a Bernoulli reward model,
decentralized MAB was also addressed in \cite{Anima,Gaiyi}, where
the single-player policy UCB1 was extended to the multi-player
setting.

The basic idea of deterministic sequencing of exploration and
exploitation was first proposed in~\cite{qingzhaosub} under the
i.i.d. reward model. To handle the restless reward model, we
introduce the epoch structure with epoch lengths carefully chosen to
achieve the logarithmic regret order. The regret analysis also
requires different techniques as compared to the i.i.d. case.
Furthermore, the extension to the decentralized setting where
different players are not required to synchronize in their epoch
structures is highly nontrivial.

The results presented in this paper and the related work discussed
above are developed within the non-Bayesian framework of MAB in
which the unknowns in the reward models are treated as deterministic
quantities and the design objective is universally (over all
possible values of the unknowns) good policies. The other line of
development is within the Bayesian framework in which the unknowns
are modeled as random variables with known prior distributions and
the design objective is policies with good average performance
(averaged over the prior distributions of the unknowns). By treating
the posterior probabilistic knowledge (updated from the prior
distribution using past observations) about the unknowns as the
system state, Bellman in 1956 abstracted and generalized the classic
Bayesain MAB to a special class of Markov decision
processes~\cite{R1}. The long-standing Bayesian MAB was solved by
Gittins in 1970s where he established the optimality of an index
policy¡ªthe so-called Gittins index policy~\cite{R2}. In 1988,
Whittle generalized the classic Bayesian MAB to the restless MAB
(with known Markovian dynamics) and proposed an index policy based
on a Lagrangian relaxation~\cite{R3}. Weber and Weiss in 1990 showed
that Whittle index policy is asymptotically optimal under certain
conditions~\cite{R4, R5}. In the finite regime, the strong
performance of Whittle index policy has been demonstrated in
numerous examples (see, e.g., \cite{R6,R7,R8,R9}).

\subsection{Applications}
The restless multi-armed bandit problem has a broad range of
potential applications. For example, in a cognitive radio network
with dynamic spectrum access~\cite{QZCR}, a secondary user searches
among several channels for idle slots that are temporarily unused by
primary users. The state of each channel (busy or idle) can be
modeled as a two-state Markov chain with unknown dynamics. At each
time, a secondary user chooses one channel to sense and subsequently
transmit if the channel is found to be idle. The objective of the
secondary user is to maximize the long-term throughput by designing
an optimal channel selection policy without knowing the traffic
dynamics of the primary users. The decentralized formulation under
the endogenous restless model applies to a network of distributed
secondary users.

The results obtained in this paper also apply to opportunistic
communication in an unknown fading environment. Specifically, each
user senses the fading realization of a selected channel and chooses
its transmission power or data rate accordingly. The reward can be
defined to capture energy efficiency (for fixed-rate transmission)
or throughput. The objective is to design the optimal channel
selection policies under unknown fading dynamics. Similar problems
under known fading models have been considered
in~\cite{AVS,AVS1,AVS2}.

Another potential application is financial investment, where a
Venture Capital (VC) selects one company to invest each year. The
state (\eg annual profit) of each company evolves as a Markov chain
with the transition matrix depending on whether the company is
invested or not~\cite{VC}. The objective of the VC is to maximize
the long-run profit by designing the optimal investment strategy
without knowing the market dynamics \emph{a priori}. The case with
multiple VCs may fit into the decentralized formulation under the
exogenous restless model.

\subsection{Notations and Organization} For two positive integers $k$ and $l$,
define $k\oslash l\defeq ((k-1)~\mbox{mod}~l)+1$, which is an
integer taking values from $1,2,\cdots,l$.

The rest of the paper is organized as follows. In Sec.~\ref{sec:CT}
we consider the single-player setting. We propose the DSEE policy
and establish its logarithmic regret order. In Sec.~\ref{sec:DT}, we
consider the decentralized setting with multiple distributed
players. We present several simulation examples in
Sec.~\ref{sec:simulation} to compare the performance of DSEE with
the policy proposed in~\cite{Tekin:10}. Sec.~V concludes this paper.

\section{The Centralized Setting}\label{sec:CT}

In this section, we consider the centralized, or equivalently, the
single-player setting. We first present the problem formulation and
the definition of regret and then propose the DSEE policy and
establish its logarithmic regret order.

\subsection{Problem Formulation}\label{sec:problemformulation}

In the centralized setting, we have one player and $N$ independent
arms. At each time, the player chooses $M$ arms to play. Each arm,
when played, offers certain amount of reward that defines the
current state of the arm. Let $s_j(t)$ and $\mathcal{S}_j$ denote,
respectively, the state of arm $j$ at time $t$ and the state space
of arm $j$. When arm $j$ is played, its state changes according to
an unknown Markovian rule with $P_j$ as the transition matrix. The
transition matrixes are assumed to be irreducible, aperiodic, and
reversible. States of passive arms transit according to an arbitrary
unknown random process. Let
$\vec{\pi}_j=\{\pi_j(s)\}_{s\in\mathcal{S}_j}$ denote the stationary
distribution of arm $j$ under $P_j$. The stationary  reward mean
$\mu_j$ is given by $ \mu_j= \sum_{s \in \mathcal{S}_j} s \pi_j(s)$.
Let $\sigma$ be a permutation of $\{1, \cdots, N\}$ such that
\[
\mu_{\sigma(1)} \geq \mu_{\sigma(2)} \geq \mu_{\sigma(3)} \geq
\cdots \geq \mu_{\sigma(N)}.
\]

A policy $\Phi$ is a rule that specifies an arm to play based on the
observation history. Let $t_j(n)$ denote the time index of the $n$th
play on arm $j$, and $T_j(t)$ the total number of plays on arm $j$
by time $t$. Notice that both $t_j(n)$ and $T_j(t)$ are random
variables with distributions determined by the policy $\Phi$.  The
total reward under $\Phi$ by time $t$ is given by
\begin{eqnarray}\label{eqn:reward1}
R(t) = \sum_{j=1}^{N}\sum_{n=1}^{T_j(t)}s_j(t_j(n)).
\end{eqnarray}

The performance of a policy $\Phi$ is measured by regret $r_{\Phi}
(t)$ defined as the reward loss with respect to the best possible
single-arm policy:
\begin{eqnarray} \label{eqn:regret}
r_\Phi(t) = t \sum_{i=1}^M \mu_{\sigma(i)} -
\mbbE_{\Phi}\left[R(t)\right] + O(1),
\end{eqnarray}
where the $O(1)$ constant term is caused by the transient effect of
playing the $M$ best arms when their initial states are not given by
the stationary distribution, $\mbbE_{\Phi}$ denotes the expectation
with respect to the random process induced by policy $\Phi$. The
objective is to minimize the growth rate of the regret with time
$t$. Note that the constant term does not affect the order of the
regret and will be omitted in the regret analysis in subsequent
sections.

In the remaining of this section, we will consider first $M=1$.
Extensions to the general case are given in Sec.~\ref{sec:II-D}.

\subsection{DSEE with An Epoch Structure}\label{sec:problemformulation}

 Compared to the i.i.d. and the rested Markovian reward models, the
restless nature of arm state evolution requires that each arm be
played consecutively for a period of time in order to learn its
Markovian reward statistics and to approach the steady state. The
length of each segment of consecutive plays needs to be carefully
controlled: it should be short enough to avoid spending too much
time on a bad arm and, at the same time, long enough to limit the
transient effect. To balance these factors, we construct a policy
based on DSEE with an epoch structure. As illustrated in
Fig.~\ref{fig:SingleUserstructure}, the proposed policy partitions
the time horizon into interleaving exploration and exploitation
epochs with geometrically growing epoch lengths. In the exploitation
epochs, the player computes the sample mean (\ie average reward per
play) of each arm based on the observations obtained so far and
plays the arm with the largest sample mean, which can be considered
as the current estimated best arm. In the exploration epochs, the
player aims to learn the reward statistics of all arms by playing
them equally many times.
 The purpose of the exploration epochs is to make decisions in the
exploitation epochs sufficiently accurate.

 As illustrated in Fig.
\ref{fig:SingleUserstructure}, in the $n$th exploration epoch, the
player plays every arm $4^{n-1}$ times. In the $n$th exploitation
epoch with length $2\times 4^{n-1}$, the player plays the arm with
the largest sample mean (denoted as arm $a^*$) determined at the
beginning of this epoch. At the end of each epoch, whether to start
an exploitation epoch or an exploration epoch is determined by
whether sufficiently many (specifically, $D\log t$ as given in
(\ref{eqn:singleplayercondition}) in
Fig.~\ref{fig:singleplayerpolicy}) observations have been obtained
from every arm in the exploration epochs. This condition ensures
that only logarithmically many plays are spent in the exploration
epochs, which is necessary for achieving  the logarithmic regret
order. This also implies that the exploration epochs are much less
frequent than the exploitation epochs. Though the exploration epochs
can be understood as the ``information gathering'' phase, and the
exploitation epochs as the ``information utilization'' phase,
observations obtained in the exploitation epochs are also used in
learning the arm reward statistics.  A complete description of the
proposed policy is given in Fig.~\ref{fig:singleplayerpolicy}.

\begin{figure}[h]
\centerline{
\begin{psfrags}
\psfrag{e}[c]{\scriptsize  Arm $1$($4^{n}$ times)}
\psfrag{ex}[c]{\scriptsize Exploration}\psfrag{ep}[c]{\scriptsize
 Exploitation} \psfrag{ ep}[c]{\scriptsize Exploitation}
\psfrag{b}[c]{\scriptsize Play the best arm $a^*$ ($2\times 4^n$
times)}  \psfrag{f}[c]{\scriptsize $\cdots$}
\psfrag{g}[c]{\scriptsize Arm $N$($4^{n}$ times)}
\scalefig{0.8}\epsfbox{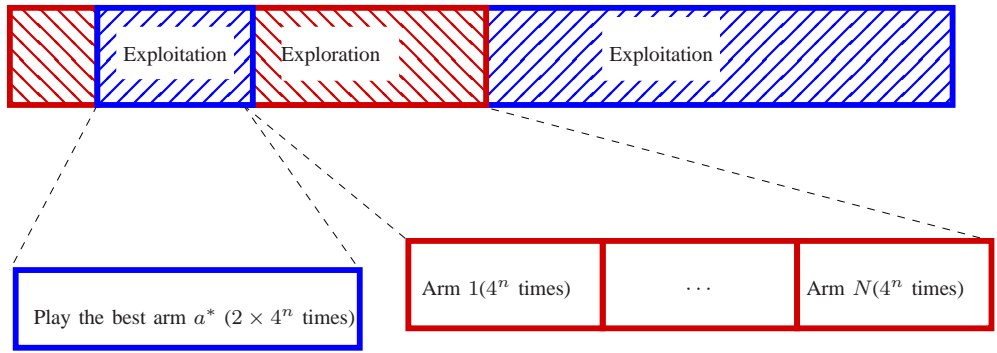} \psfrag{Two types}[c]{\scriptsize
Two types of epochs are not always interleaving.}
\end{psfrags}}
\caption{The epoch structure with geometrically growing epoch
lengths.} \label{fig:SingleUserstructure}
\end{figure}

\begin{figure}[htbp]
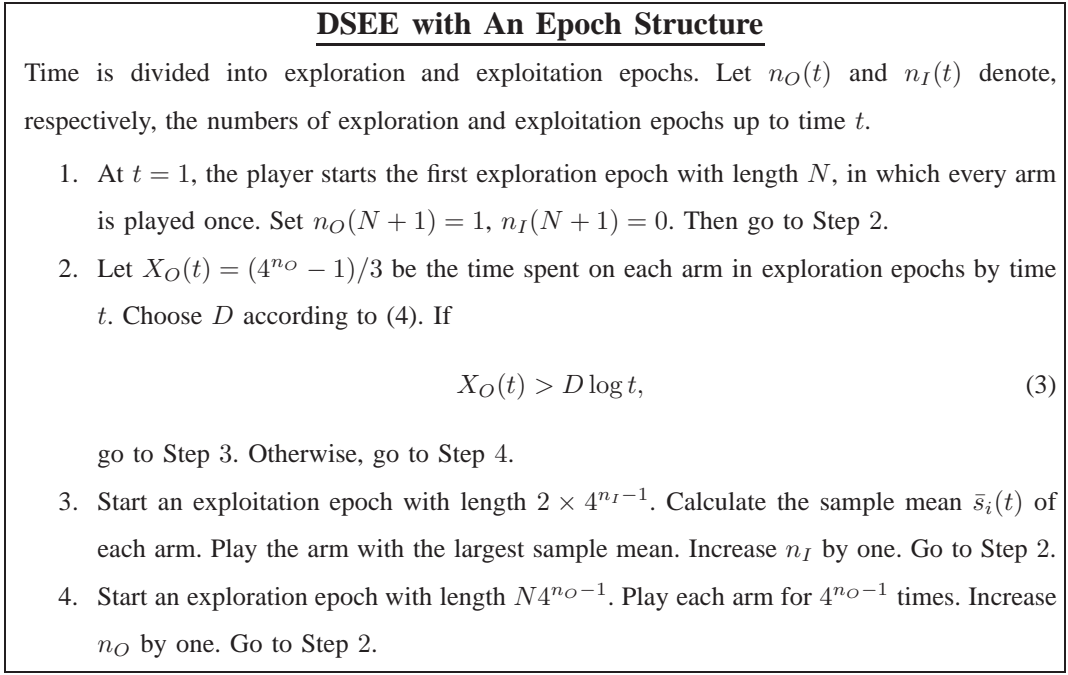

\begin{center}
\noindent\fbox{
\parbox{5.4in}
{ \centerline{\underline{{\bf DSEE with An Epoch Structure}}}
{\small Time is divided into exploration and exploitation epochs.
Let $n_O(t)$ and $n_I(t)$ denote, respectively, the numbers of
exploration and exploitation epochs up to time $t$.
\begin{enumerate}
\item[1.] At $t=1$, the player starts the first exploration epoch with length $N$, in which every arm is played once.  Set $n_O(N+1) = 1$, $n_I(N+1) = 0$. Then go to Step $2$.
\item[2.] Let $X_O(t) = (4^{n_O}-1)/3$ be the time spent on each arm in exploration epochs by time
$t$. Choose $D$ according to (\ref{eqn:conditionD1}). If
\begin{eqnarray}\label{eqn:singleplayercondition}
X_O(t)
> D \log t ,
\end{eqnarray} go to Step $3$. Otherwise, go to
Step $4$.
\item[3.] Start an exploitation epoch with length $2\times 4^{n_I - 1}$. Calculate the sample mean  $\bar{s}_i(t)$ of each arm.  Play the arm with the largest
sample mean. Increase $n_I$ by one. Go to Step $2$.
\item[4.] Start an exploration epoch with length $N 4^{n_O-1}$. Play each arm  for $4^{n_O-1}$ times. Increase $n_O$ by one. Go to Step $2$.
\end{enumerate}}
}} \caption{DSEE with an epoch structure for
RMAB.}\label{fig:singleplayerpolicy}
\end{center}
\end{figure}

\subsection{Regret Analysis} \label{sec:Performance}

In this section, we show that  the proposed policy achieves a
logarithmic regret order. This is given in the following theorem.
\begin{theorem}\label{thm:singleuser}
Assume that $\{P_i\}_{i=1}^N$ are finite state, irreducible,
aperiodic and reversible. All the reward states  are non-negative.
Let $\epsilon_i$ be the second largest eigenvalue of $P_i$. Define
$\epsilon_{\min} = \min_{1\leq i \leq N} \epsilon_i$, $\pi_{\min} =
\min_{ 1\leq i \leq N, s\in \mathcal{S}_i} \pi_i(s)$,
 $r_{\max} =
\max_{1 \leq i \leq N} \sum_{s \in \mathcal{S}_i} s$,
$|\mathcal{S}|_{\max} = \max_{1\leq i \leq N} |\mathcal{S}_i|$,
$A_{\max} = \max_i(\min_{s\in \mathcal{S}_i} \pi^i_s)^{-1}\sum_{s\in
\mathcal{S}_i}s$, and $L = \frac{30 r_{\max}^2 }{( 3 -
2\sqrt{2})\epsilon_{\min}}$. Assume that the best arm has a distinct
reward mean\footnote{The extension to the general case is
straightforward}. Set the policy parameters $D$ to satisfy the
following condition:
\begin{eqnarray}\label{eqn:conditionD1}
 D & \geq & \frac{4L}{(\mu_{\sigma(1)} - \mu_{\sigma(2)})^2},
\end{eqnarray}
 The regret of DSEE at the end of any epoch can be upper
bounded by
\begin{eqnarray}\label{eqn:totalregretsingle} \nonumber
r_{\Phi}(t)& \leq & C_1 \lceil \log_4 (\frac{3}{2} (t- N) +1) \rceil
 + C_2 [4(3D \log t +1)-1] \\
&&+ N A_{\max}(\lfloor \log _4 (3D \log t +1) \rfloor + 1) ),
\end{eqnarray}
where
\begin{eqnarray}
C_1 & = &  \left( A_{\max}+ 3 \sum_{j= 2}^N \frac{ \mu_{\sigma(1)} -
\mu_{\sigma(j)} }{ \pi_{\min} }\sum_{ k = 1,j} \left( \frac{1}{\log
2} + \frac{\sqrt{2}\epsilon_k \sqrt{L }}{10 \sum_{s \in {S}_k} s}
\right) |\mathcal{S}_k| \right),\\
 C_2 &  = & \frac{1}{3}\left(N \mu_{\sigma(1)}
- \sum_{i=1}^{N}\mu_{\sigma(i)}\right).
\end{eqnarray}
\end{theorem}
\begin{proof}
\vspace{1.0em}
 See Appendix~A for details.
\end{proof}

 In the proposed policy, to ensure the logarithmic regret
order, the policy parameter $D$ needs to be chosen appropriately.
This requires an arbitrary (but nontrivial) bound on $r_{\max}$,
$\epsilon_{\min}$, and $\mu_{\sigma(1)} - \mu_{\sigma(2)}$. In the
case where no knowledge about the system is available, $D$ can be
chosen to increase with time rather than set \emph{a priori} to
achieve a regret order arbitrarily close to logarithmic. This is
formally stated in the following theorem.

\begin{theorem}\label{thm:singleplayerarbitraryclose}
Assume that $\{P_i\}_{i=1}^N$ are finite state, irreducible,
aperiodic and reversible. All the reward states  are non-negative.
For any increasing sequence $f(t)$ ($f(t) \rightarrow \infty $ as $t
\rightarrow \infty$), if $D(t) = f(t)$, then
\begin{eqnarray}
r_{\Phi}(t) \sim O(f(t)\log t).
\end{eqnarray}
\end{theorem}

\begin{proof}
See Appendix~B for details.
\end{proof}

\subsection{Extension to $M>1$} \label{sec:II-D}

For $M>1$, the basic structure of DSEE is the same. The only
difference is that in the $n$th exploitation epoch with length
$2\times 4^{n-1}$, the player plays the arms with the $M$ largest
sample means; in the $n$th exploration epoch with length  $ \lceil
\frac{N}{M} \rceil 4^{n-1}$  the player spends $4^{n-1} $ plays on
each arm and gives up $\left( M\lceil \frac{N}{M}  \rceil-N \right)
4^{n-1} $ plays. The  regret in this case is given in the following
theorem.

\begin{theorem}\label{thm:singleuserM}
Assume that $\{P_i\}_{i=1}^N$ are finite state, irreducible,
aperiodic and reversible. All the reward states  are non-negative.
Let $\epsilon_i$ be the second largest eigenvalue of $P_i$. Define
$\epsilon_{\min} = \min_{1\leq i \leq N} \epsilon_i$, $\pi_{\min} =
\min_{ 1\leq i \leq N, s\in \mathcal{S}_i} \pi_i(s)$,
 $r_{\max} =
\max_{1 \leq i \leq N} \sum_{s \in \mathcal{S}_i} s$,
$|\mathcal{S}|_{\max} = \max_{1\leq i \leq N} |\mathcal{S}_i|$,
$A_{\max} = \max_i(\min_{s\in \mathcal{S}_i} \pi^i_s)^{-1}\sum_{s\in
\mathcal{S}_i}s$, and $L = \frac{30 r_{\max}^2 }{( 3 -
2\sqrt{2})\epsilon_{\min}}$. Assume that the $M$th best arm arm has
a distinct reward mean\footnote{The extension to the general case is
straightforward}. Set the policy parameters $D$ to satisfy the
following condition:
\begin{eqnarray}\label{eqn:conditionD1M}
 D & \geq & \frac{4L}{(\mu_{\sigma(M)} - \mu_{\sigma(M+1)})^2},
\end{eqnarray}
 The regret of DSEE at the end of any epoch can be upper
bounded by
\begin{eqnarray}\label{eqn:totalregretsingleM} \nonumber
r_{\Phi}(t)& \leq & C_1 \lceil \log_4 (\frac{3}{2} (t- N) +1) \rceil
 + C_2 [4(3D \log t +1)-1] \\
&&+ N A_{\max}(\lfloor \log _4 (3D \log t +1) \rfloor + 1) ,
\end{eqnarray}
where
\begin{eqnarray}
C_1 & = & M  A_{\max}+ \frac{3}{ \pi_{\min}}\sum_{j = 1}^M \sum_{i=
M+1}^N( \mu_{\sigma(j)} - \mu_{\sigma(i)} )\sum_{ k = j,i} \left(
\frac{1}{\log 2} + \frac{\sqrt{2}\epsilon_k \sqrt{L }}{10 \sum_{s
\in {S}_k} s} \right)
|\mathcal{S}_k|,\\
 C_2 &  = & \frac{1}{3} \left( \lceil \frac{N}{M} \rceil
\sum_{i = 1}^M \mu_{\sigma(i)} -
\sum_{i=1}^{N}\mu_{\sigma(i)}\right).
\end{eqnarray}
\end{theorem}
\begin{proof}
\vspace{2.0em}
 See Appendix~C for details.
\end{proof}

Regret at any time $t$ has a upper bound with a logarithmic order
similar to (\ref{eqn:totalregretsingleM}), with $t$ replaced by
$4t+3$. In the proposed policy, to ensure the logarithmic regret
order, the policy parameter $D$ needs to be chosen appropriately.
This requires an arbitrary (but nontrivial) bound on $r_{\max}$,
$\epsilon_{\min}$, and $\mu_{\sigma(M)} - \mu_{\sigma(M+1)}$. In the
case where no knowledge about the system is available, $D$ can be
chosen to increase with time rather than set \emph{a priori} to
achieve a regret order arbitrarily close to logarithmic. This is
formally stated in the following theorem.

\begin{theorem}\label{thm:singleplayerarbitrarycloseM}
Assume that $\{P_i\}_{i=1}^N$ are finite state, irreducible,
aperiodic and reversible. All the reward states  are non-negative.
For any increasing sequence $f(t)$ ($f(t) \rightarrow \infty $ as $t
\rightarrow \infty$), if $D(t) = f(t)$, then
\begin{eqnarray}
r_{\Phi}(t) \sim O(f(t)\log t).
\end{eqnarray}
\end{theorem}

\begin{proof}
See Appendix~B for details.
\end{proof}

\section{The Decentralized Setting} \label{sec:DT}

\subsection{Problem Formulation}
In the decentralized setting, there are $M$ players and $N$
independent arms ($N>M$). At each time, each player chooses one arm
to play based on its local observations. As in the single player
case, the reward state of arm $j$ changes according to a Markovian
rule when played, and the same set of notations are adopted. For the
state transition of a passive arm, we consider two models: the
endogenous restless model and the exogenous restless model. In the
former, the arm evolves according to an arbitrary unknown random
process even when it is not played. In the latter, the system itself
is rested. From each individual player's perspective, however, arms
are restless due to actions of other players that are unobservable
and uncontrollable. The players do not know the arm dynamics and do
not communicate with each other. Collisions occur when multiple
players select the same arm to play. Different collision models can
be adopted, where the players in conflict can share the reward or no
one receives any reward. In the former, the total reward under a
policy $\Phi$ by time $t$ is given by
\begin{eqnarray}
R(t) = \sum_{j=1}^{N}\sum_{n=1}^{T_j(t)}s_j(t_j(n))
\mathbb{I}_j({t_j(n)}),
\end{eqnarray}
where for the case conflicted players share the reward
$\mathbb{I}_j({t_j(n)}) = 1$ if arm $j$ is played at least one
player at time $t_j(n)$, and $\mathbb{I}_j({t_j(n)}) = 0$ otherwise;
for the case conflicted players get no reward
$\mathbb{I}_j({t_j(n)}) = 1$ if arm $j$ is played one and only one
player at time $t_j(n)$, and $\mathbb{I}_j({t_j(n)}) = 0$ otherwise.

Under both restless models and both collision models, regret
$r_{\Phi} (t)$ is defined as the reward loss with respect to the
ideal scenario of a perfect orthoganalization of the $M$ players
over the $M$ best arms. We thus have
\begin{eqnarray} \label{eqn:regret}
r_\Phi(t) = t\sum_{i=1}^{M}\mu_{\sigma(i)} - \mbbE_{\Phi}R(t) +
O(1),
\end{eqnarray}
where the $O(1)$ constant term comes from the transient effect of
the $M$ best arms (similar to the single-player setting). Note that
under the exogenous restless model, this definition of regret is
strict in the sense that $t\sum_{i=1}^{M}\mu_{\sigma(i)} + O(1) $ is
indeed the maximal expected reward achievable under a known model of
the arm dynamics.

\subsection{Decentralized DSEE Policy}\label{sec:policy}

For the ease of presentation, we first assume that the players are
synchronized according to a global time. Since the epoch structure
of DSEE is deterministic, global timing ensures synchronized
exploration and exploitation among players. We further assume that
the players have pre-agreement on the time offset for sharing the
arms, determined based on, for example, the players' ID. We will
show in Sec.~\ref{sec:elimination} that this requirement on global
timing and pre-agreement can be eliminated  to achieve a complete
decentralization.

The decentralized DSEE has a similar epoch structure. In the
exploration epochs (with the $n$th one having length $ N \times
4^{n-1}$), the players play all $N$ arms in a round-robin fashion
with different offsets determined in the pre-agreement. In the
exploitation epochs, each players calculates the sample mean of
every arm based on its own local observations and plays the arms
with the $M$ largest sample mean in a round-robin fashion with a
certain offset. Note that even though the players have different
time-sharing offsets, collisions occur during exploitation epochs
since the players may arrive at different sets and ranks of the $M$
arms due to the randomness in their local observations. Each of
these $M$ arms is played $2\times 4^{n-1}$ times. The $n$th
exploitation epoch thus has length $2M\times 4^{n-1}$. A detailed
description of the decentralized DSEE policy is given in
Fig.~\ref{fig:decenpolicy}.

\begin{figure}[htbp]
\begin{center}
\noindent\fbox{
\parbox{5.4in}
{ \centerline{\underline{{\bf Decentralized DSEE}}} {\small Time is
divided into  exploration and exploitation epochs with $n_O(t)$ and
$n_I(t)$ similarly defined as in Fig.~\ref{fig:singleplayerpolicy}.
\begin{enumerate}
\item[1.] At $t=1$, each player starts the first exploration epoch with length $N$. Player $k$ plays arm $(k + t) \oslash N$ at time $t$. Set $n_O(N+1) = 1$, $n_I(N+1) = 0$. Then go to Step $2$.
\item[2.] Let $X_O(t) = (4^{n_O}-1)/3$ be the time spent on each arm in exploration epochs by time
$t$. Choose $D$ according to (\ref{eqn:conditionD}). If
\begin{eqnarray}\label{eqn:condition}
X_O(t)
> D \log t ,
\end{eqnarray} go to Step $3$. Otherwise, go to
Step $4$.
\item[3.] Start an exploitation epoch with length $2M \times 4^{n_I-1}$.
Calculate sample mean $\bar{s}_i(t)$ of each arm and denote the arms
with the $M$ largest sample means as arm $a^*_{1}$ to arm $a^*_{M}$.
Each exploitation epoch is divided into $M$ subepochs with each
having a length of $2 \times 4^{n_I-1}$. Player $k$ plays arm
$a^*_{(k+m)\oslash M}$ in the $m$th subepoch. Increase $n_I$ by one.
Go to step $2$.
\item[4.] Start an exploration epoch with length $N \times 4^{n_O-1}$.
 Each exploration epoch is divided into $N$ subepochs
with each having a length of $4^{n_O-1}$. Player $k$ plays arm
$a_{(m+k) \oslash N}$ in the $m$th subepoch. Increase $n_O$ by one.
Go to step $2$.
\end{enumerate}}
}} \caption{Decentralized DSEE policy for
RMAB.}\label{fig:decenpolicy}
\end{center}
\end{figure}

\subsection{Regret Analysis} \label{sec:Performance}
In this section, we show that the decentralized DSEE policy achieves
the same logarithmic regret order as in the centralized setting.
\begin{theorem}\label{thm:multipleplayers}
Under the same notations and definitions as  in Theorem
\ref{thm:singleuser}, assume that different arms have different mean
values\footnote{This assumption can be relaxed when the players
determine the round-robin order of the arms based on pre-agreed arm
indexes rather than the estimated arm rank. This assumption is only
for simplicity of the presentation.}. Set the policy parameter $D$
to satisfy the following condition:
\begin{eqnarray}\label{eqn:conditionD}
D & \geq & \frac{4L}{(\min_{j\leq M}( \mu_{\sigma(j)} -
\mu_{\sigma(j+1)}))^2}.
\end{eqnarray}
Under sharing reward conflict model, the regret of the decentralized
DSEE at the end of any epoch can be upper bounded by
\begin{eqnarray}\label{eqn:2}
r_{\Phi}(t) \leq  C_1\lceil \log_4 (\frac{3t}{2M} +1) \rceil  +
C_2(\lfloor \log _4 (3D \log t +1) \rfloor + 1)  +C_3[4(3D \log t
+1)-1],
\end{eqnarray}
where
\begin{eqnarray}
C_1  & = &  \left\{ \begin{array}{rl } &  \sum_{m = 1}^M
\mu_{\sigma(m)} \frac{3M}{ \pi_{\min}}  \sum_{j = 1}^{M} \sum_{i =
1, i \neq j}^{N} \sum_{ k = i,j} \left( \frac{1}{\log 2} +
\frac{\sqrt{2}\epsilon_k \sqrt{L }}{10 \sum_{s \in {S}_k} s} \right)
|\mathcal{S}_k| +M^2 A_{\max}, \\ & \qquad   \qquad \qquad \qquad
\text{
Endogenous restless and  zero-reward collision model} \\
&  \frac{3M}{ \pi_{\min}}  \sum_{j = 1}^{M}\mu_{\sigma(j)} \sum_{i =
1, i \neq j}^{N} \sum_{ k = i,j} \left( \frac{1}{\log 2} +
\frac{\sqrt{2}\epsilon_k \sqrt{L }}{10 \sum_{s \in {S}_k} s} \right)
|\mathcal{S}_k| +M^2 A_{\max}, \\ &   \qquad \qquad \qquad \qquad
\text{ Endogenous restless and partial-reward collision model}
\\ & \sum_{m = 1}^M \mu_{\sigma(m)} \frac{3M}{ \pi_{\min}}  \sum_{j = 1}^{M}
\sum_{i = 1, i \neq j}^{N} \sum_{ k = i,j} \left( \frac{1}{\log 2} +
\frac{\sqrt{2}\epsilon_k \sqrt{L }}{10 \sum_{s \in {S}_k} s} \right)
|\mathcal{S}_k|, \\  &  \qquad   \qquad  \qquad \qquad
 \text{ Exogenous restless and zero-reward collision  model }
 \\ &  \frac{3M}{ \pi_{\min}}  \sum_{j = 1}^{M} \mu_{\sigma(j)}
\sum_{i = 1, i \neq j}^{N} \sum_{ k = i,j} \left( \frac{1}{\log 2} +
\frac{\sqrt{2}\epsilon_k \sqrt{L }}{10 \sum_{s \in {S}_k} s} \right)
|\mathcal{S}_k|, \\  &  \qquad   \qquad  \qquad \qquad
 \text{ Exogenous restless  and partial-reward collision model}\end{array} \right. \\
C_2 & = & \left\{ \begin{array}{rl }  NM A_{\max},& \text{Endogenous
 restless model}
\\ 0, & \text{Exogenous  restless model}
\end{array} \right. \\
C_3  & = &
 \frac{1}{3} \left(N
\sum_{i=1}^{M}\mu_{\sigma(i)} - M
\sum_{i=1}^{N}\mu_{\sigma(i)}\right).
  \end{eqnarray}

\end{theorem}

 \vspace{1.0em}
\begin{proof}
See Appendix~D for details.
\end{proof}

Achieving the logarithmic regret order requires an arbitrary (but
nontrivial) bound on $r_{\max}$, $\min_{j\leq M}( \mu_{\sigma(j)} -
\mu_{\sigma(j+1)})$, and $\epsilon_{\min}$.  Similarly to the
single-player case, $D$ can be chosen to increase with time to
achieve a regret order arbitrarily close to logarithmic as stated
below.

\begin{theorem}\label{thm:multipleplayersarbitraryclose}
Under the same notations and definitions as  in Theorem
\ref{thm:singleuser}, assume that different arms have different mean
values. For any increasing sequence $f(t)$ ($f(t) \rightarrow \infty
$ as $t \rightarrow \infty$), if $D(t)$ is chosen such that $D(t) =
f(t)$, then under both the endogeneous and exogeneous restless
models,
\begin{eqnarray}
r_{\Phi}(t) \sim O(f(t)\log t).
\end{eqnarray}
\end{theorem}
\begin{proof}
See Appendix~E for details.
\end{proof}

\subsection{In the Absence of Global Synchronization and Pre-agreement}\label{sec:elimination}

In this section, we show that the requirement on global
synchronization and pre-agreement can be eliminated while
maintaining the logarithmic order of the policy. As a result,
players can join the system at different times.

Without global timing and pre-agreement, each player has its own
exploration and exploitation epoch timing. The epoch structure of
each player's local policy is similar to that given in
Fig.~\ref{fig:decenpolicy}. The only difference is that in each
exploitation epoch, instead of playing the top $M$ arms (in terms of
sample mean) in a round-robin fashion, the player randomly and
uniformly chooses one of them to play. When a collision occurs
during the exploitation epoch, the player makes another random and
uniform selection among the top $M$ arms. As shown in the proof of
Theorem~\ref{thm:5}, this simple adjustment based on collisions
achieves efficient sharing among all players without global
synchronization and pre-agreement. Note that during an exploration
epoch, the player plays all $N$ arms in a round-robin fashion
without reacting to collisions. Since the players still observe the
reward state of the chosen arm, collisions affect only the immediate
reward but not the learning ability of each player. As a
consequence, collisions during a player's exploration epochs will
not affect the logarithmic regret order since the total length of
exploration epochs is at the logarithmic order. The key to
establishing the logarithmic regret order in the absence of global
synchronization and pre-agreement is to show that collisions during
each player's exploitation epochs are properly bounded and efficient
sharing can be achieved.

\begin{theorem}\label{thm:5}
Under the same notations and definitions as in Theorem
\ref{thm:singleuser}, Decentralized DSEE without global
synchronization and pre-agreement achieves logarithmic regret order.
\end{theorem}
\begin{proof}
See Appendix~F for details.
\end{proof}

The assumption that the arm reward state is still observed when
collisions occur holds in many applications. For example, in the
applications of dynamic spectrum access and opportunistic
communications under unknown fading, each user first senses the
state (busy/idle or the fading condition) of the chosen channel
before a potential transmission. Channel states are always observed
regardless of collisions. The problem is much more complex when
collisions are unobservable and each player involved in a collision
only observes its own local reward (which does not reflect the
reward state of the chosen arm). In this case, collisions result in
corrupted measurements that cannot be easily screened out, and
learning from these corrupted measurements may lead to misidentified
arm rank. How to achieve the logarithmic regret order without global
timing and pre-agreement in this case is still an open problem.

\section{Simulation Results} \label{sec:simulation}

In this section, we study the performance of DSEE as compared to the
RCA policy proposed in~\cite{Tekin:10-2}. The first example is in
the context of cognitive radio networks. We consider that a
secondary user searches for idle channels unused by the primary
network. Assume that the spectrum consists of $N$ independent
channels. The state---busy ($0$) or idle ($1$)---of each channel
(say, channel $n$) evolves as a Markov chain with transition
probabilities $\{p_{ij}^n\} ~ i,j \in \{0,1\}$. At each time, the
secondary user selects a channel to sense and choose the
transmission power according to the channel state. The reward
obtained from a transmission over channel $n$ in state $i$ is given
by $r_i^n$. We use the same set of parameters chosen
in~\cite{Tekin:10} (given in the caption of Fig.~\ref{fig:Simu1}).
We observe from Fig.~\ref{fig:Simu1} that RCA initially outperforms
DSEE for a short period, but DSEE offers significantly better
performance as time goes, and the regret offered by RCA does not
seem to converge to the logarithmic order in a horizon of length
$10^4$. We also note that while the condition on the policy
parameter $D$ given in (\ref{eqn:conditionD1}) is sufficient for the
logarithmic regret order, it is not necessary. Fig.~\ref{fig:Simu1}
clearly shows the convergence to the logarithmic regret order for a
small value of D, which leads to better finite-time performance.

\begin{figure}[h]
\centerline{
\begin{psfrags}
\scalefig{0.8}\epsfbox{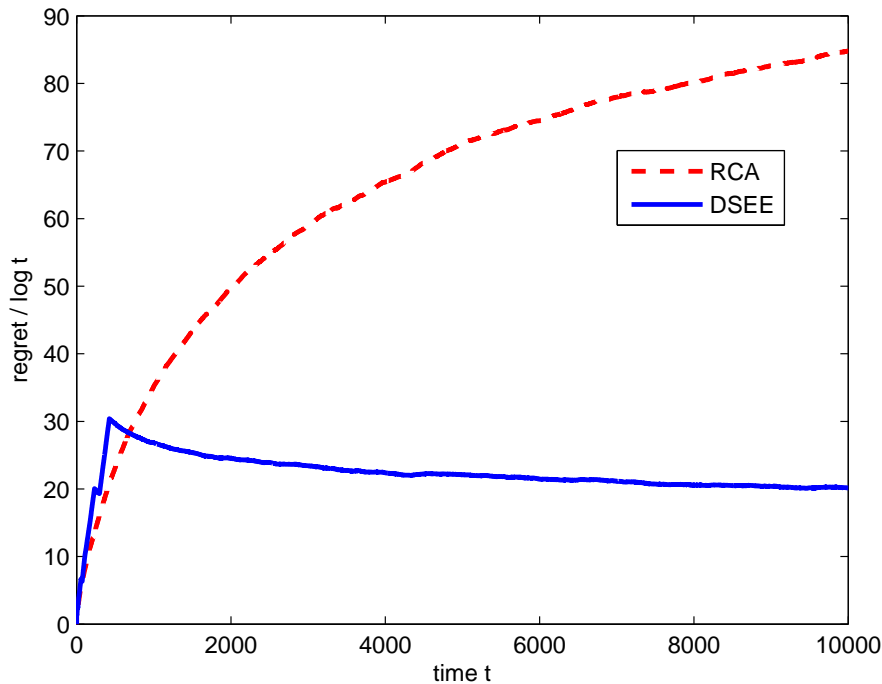}
\end{psfrags}}
\caption{Regret for DSEE and RCA, $p_{01} =
 [0.1, 0.1, 0.5,
0.1, 0.1]$, $p_{10} = [ 0.2, 0.3, 0.1, 0.4, 0.5]$, $r_1 =
[1,1,1,1,1]$, $r_0 = [0.1,0.1,0.1,0.1,0.1]$, $D = 10$, $L = 10$,
$100$ Monte Carlo runs.}\label{fig:Simu1}
\end{figure}

In the next example, we consider a case with a relatively large
reward state space. We consider a case with  $5$ arms, each having
$20$ states. Rewards from each state for arm~$2$ to arm~$5$ is
$[1,2,\cdots 20]$. Rewards from each state for arm~$1$ is
$1.5\times[1,2,\cdots 20]$ (to make it a better arm than the rest).
Transition probabilities  of all arms were generated randomly and
can be found in Appendix~G. The stationary distributions of all arms
are close to uniform, which avoids the most negative effect of
randomly chosen pilot states in RCA. The values of $D$ in DSEE and
$L$ in RCA were chosen to be the minimum as long as the ratio of the
regret to $\log t$ converges to a constant with a reasonable time
horizon. We again observe a better performance from DSEE as shown in
Fig.~\ref{fig:Simu2}.

\begin{figure}[h]
\centerline{
\begin{psfrags}
\scalefig{0.8}\epsfbox{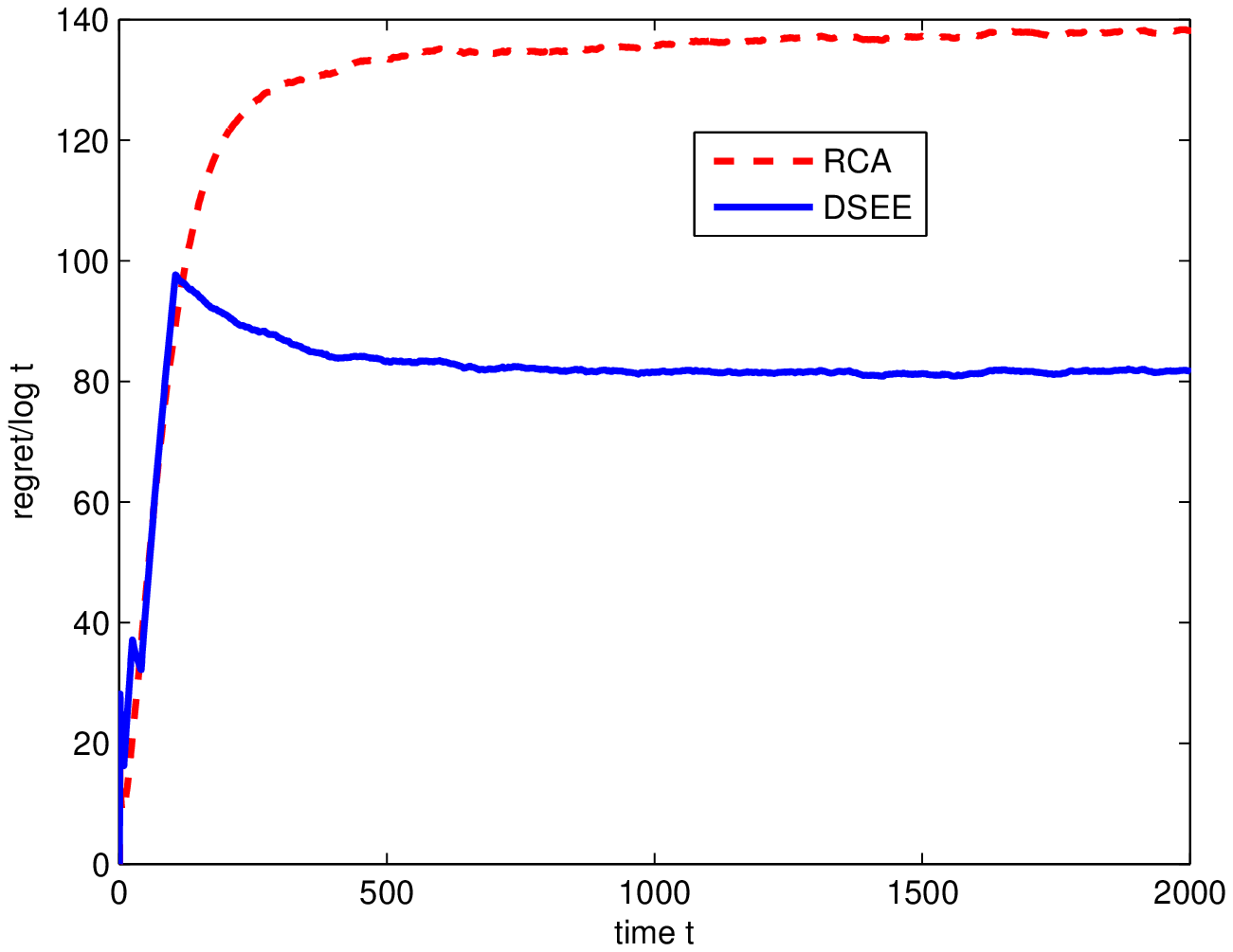}
\end{psfrags}}
\caption{Regret for DSEE and RCA with $5$ arms, $20$ states, $L =
20$, $D = 1.8$, $1000$ Monte Carlo runs. }\label{fig:Simu2}
\end{figure}

The better performance of DSEE over RCA may come from the fact that
DSEE learns from all observations while RCA only uses observations
within the regenerative cycles in learning. When the arm reward
state space is large or the randomly chosen pilot state that defines
the regenerative cycle has a small stationary probability, RCA may
have to discard a large number of observations from learning.

\section{Conclusion} \label{sec:conclusion}

In this paper, we studied the restless multi-armed bandit (RMAB)
problem with unknown dynamics under both centralized (single-player)
and decentralized settings. We developed a policy based on a
deterministic sequencing of exploration and exploitation with
geometrically growing epochs that achieves the logarithmic regret
order. In particular, in the decentralized setting with multiple
distributed players, the proposed policy achieves a complete
decentralization for both the exogenous and endogenous restless
models.

\section*{Appendix A. Proof of Theorem~\ref{thm:singleuser} and Theorem~\ref{thm:singleuserM}}\label{sec:AppendixA}

We first rewrite the definition of regret as
\begin{eqnarray}
r_{\Phi}(t) & = & t\mu_{\sigma(1)}  - \mbbE_{\Phi}R(t)
\\ & = & \sum_{i=1}^{N} [\mu_i \mbbE[T_i(t)] -
\mbbE[\sum_{n=1}^{T_i(t)}s_i(t_i(n))]] + \left[ t\mu_{\sigma(1)}-
\sum_{i=1}^{N} \mu_i \mbbE[T_i(t)]\right]. \label{eqn:regreform}
\end{eqnarray}
To show that the regret has a logarithmic order, it is sufficient to
show that the two terms in (\ref{eqn:regreform})
 have logarithmic orders. The first term in (\ref{eqn:regreform}) can be considered as the
regret caused by transient effect. The second term can be considered
as the regret caused by engaging a bad arm. First, we bound the
regret caused by transient effect based on the following lemma.

{\em Lemma $1$}~\cite{Anantharam:87-2}: Consider an irreducible,
aperiodic Markov chain with state space $\mathcal{S}$, transition
probabilities $P$, an initial distribution $\vec{q}$ which is
positive in all states, and stationary distribution $\vec{\pi}$
($\pi_s$ is the stationary probability of state $s$). The state
(reward) at time $t$ is denoted by $s(t)$. Let $\mu$ denote the mean
reward. If we play the chain for an arbitrary time $T$, then there
exists a value $A_P \leq (\min_{s\in \mathcal{S}} \pi_s)^{-1}
\sum_{s\in \mathcal{S}} s $ such that $\mbbE[\sum_{t=1}^{T}s(t) -\mu
T] \leq A_P.$

Lemma~$1$ shows that if the player continues to play an arm for time
$T$, the difference between the expected reward and $ T \mu$ can be
bounded by a constant that is independent of $T$. This constant is
an upper bound for the regret caused by each arm switching. If there
are only logarithmically many arm switchings as times goes, the
regret caused by arm switching has a logarithmic order.  An upper
bound on the number of arm switchings is shown below. It is
developed by bounding the numbers of the exploration epochs and the
exploitation epochs respectively.

For the exploration epochs, by time $t$, if the player has started
the $(n+1)$th exploration epoch, we have
\begin{eqnarray}
\frac{1}{3}(4^n-1) < D \log t,
\end{eqnarray}
 where $\frac{1}{3}(4^n-1) $ is the time spent on each arm in the first $n$
exploration epochs. Consequently the number of the exploration
epochs can be bounded by
\begin{eqnarray}\label{eqn:ubexploration}
n_O(t) \leq \lfloor \log _4 (3D \log t +1) \rfloor + 1.
\end{eqnarray}
By time $t$, at most $(t-N)$ time slots have been spent on the
exploitation epochs. Thus
\begin{eqnarray}\label{eqn:ubexploitation}
n_I(t) \leq \lceil \log_4 (\frac{3}{2} (t- N) +1) \rceil .
\end{eqnarray}
Hence an logarithmic upper bound of the first term in
(\ref{eqn:regreform}) is
\begin{eqnarray}\label{eqn:firstterm}
\sum_{i=1}^{N} [\mu_i \mbbE[T_i(t)] -
\mbbE[\sum_{n=1}^{T_i(t)}s_i(t_i(n))]] \leq (\lceil \log_4
(\frac{3}{2} (t- N) +1) \rceil
 +N(\lfloor \log _4 (3D \log t +1) \rfloor + 1) )A_{\max}.
\end{eqnarray}

Next we show that the second term of (\ref{eqn:regreform}) has a
logarithmic order by bounding the total time spent on the bad arms.
We first bound the time spent on the bad arms during the exploration
epochs. Let $T_{O}(t)$ denote the time spent on each arm in the
exploration epochs by time $t$. By (\ref{eqn:ubexploration}), we
have
\begin{eqnarray}\label{eqn:exploration}
T_{O}(t) \leq \frac{1}{3}[4(3D \log t +1)-1].
\end{eqnarray}
Thus regret caused by playing bad arms in the exploration epochs is
\begin{eqnarray}\label{eqn:MPRploration}
\frac{1}{3}[4(3D \log t +1)-1]   \left(N \mu_{\sigma(1)} -
\sum_{i=1}^{N}\mu_{\sigma(i)}\right).
\end{eqnarray}

Next, we bound the time spent on the bad arms during the
exploitation epochs. Let $t_n$ denote the starting point of the
$n$th exploitation epoch. Let $\Pr[i,j,n]$ denote the probability
that arm $i$ has a larger sample mean than arm $j$ at $t_n$ when arm
$j$ is the best arm, \ie $\Pr[i,j,n]$ is the probability of making a
mistake in the $n$th exploitation epoch. Let $w_i$ and $w_j$ denote,
respectively, the number of plays on arm $i$ and arm $j$ by $t_n$.
Let $C_{t,w}=\sqrt{(L\log t / w)}$. We have
\begin{eqnarray}
\Pr[i,j,n]& \leq & \Pr [ \bar{s}_i(t_n) \geq \bar{s}_{j}(t_n) ]
 \\ \label{eqn:twothreeevents}
& \leq & \Pr [ \bar{s}_{j}(t_n) \leq \mu_j - C_{t_n,w_j}]+
\Pr[\bar{s}_{i}(t_n)\geq
\mu_i + C_{t_n,w_i}] \nonumber \\
& & \quad \quad + \Pr [ \mu_j < \mu_i + C_{t_n,w_i} + C_{t_n,w_j}] )
\\ \label{eqn:twoevents}& \leq & \Pr [ \bar{s}_{j}(t_n)
\leq \mu_j - C_{t_n,w_j}]+ \Pr[\bar{s}_{i}(t_n)\geq \mu_i +
C_{t_n,w_i}],
\end{eqnarray}
where (\ref{eqn:twoevents}) follows from the fact that $w_i \geq
D\log t_n$ and $w_j \geq D\log t_n$ and the condition on $D$ given
in (\ref{eqn:conditionD1}).

Next we bound the two quantities in (\ref{eqn:twoevents}). Consider
first the second term $\Pr[\bar{s}_{i}(t_n)\geq \mu_i + C_{t_n,w_i}]
= \Pr[w_i \bar{s}_{i}(t_n)\geq w_i\mu_i + \sqrt{L w_i \log t_n}]$.
Note that the total $w_i$ plays on arm $i$ consists of multiple
contiguous segments of the Markov sample path, each in a different
epoch. Let $K$ denote the number of such segments. From the
geometric growth of the epoch lengths, we can see that the length of
each segment is in the form of $2^{k_l}$ ( $l  = 1, \cdots, K$) with
$k_l$'s being distinct. Without loss of generality, let $k_1 < k_2 <
\cdots < k_K$. Note that $w_i$, $K$, and $k_l$'s are random
variables. The derivation below holds for every realization of these
random variables. Let $R_i(l)$ denote the total reward obtained
during the $l$th segment. Notice that $w_i = \sum_{l =
1}^{K}2^{k_l}$ and $\sqrt{w_i} \geq \sum_{l = 1}^{K}(\sqrt{2} -
1)\sqrt{2^{k_l}}$. We then have
\begin{eqnarray}
& & \Pr \left[ w_i \bar{s}_{i}(t_n)\geq w_i\mu_i + \sqrt{L w_i \log t_n} \right] \\
& \leq & \Pr \left[ \sum_{l = 1}^{K} R_i(l)\geq \mu_i\sum_{l =
1}^{K}2^{k_l} + \sqrt{L\log t_n}
 (\sqrt{2} - 1) \sum_{l = 1}^{K}\sqrt{2^{k_l}} \right] \\
 & = & \Pr \left[ \sum_{l = 1}^{K} R_i(l)-  \mu_i\sum_{l =
1}^{K}2^{k_l} - \sqrt{L\log t_n}
 (\sqrt{2} - 1) \sum_{l = 1}^{K}\sqrt{2^{k_l}} \geq 0 \right] \\
  & =  & \Pr \left[ \sum_{l = 1}^{K} \left( R_i(l)-  \mu_i 2^{k_l} - \sqrt{L\log t_n}
 (\sqrt{2} - 1) \sqrt{2^{k_l}} \right) \geq 0 \right] \\
  &  \leq  & \sum_{l = 1}^{K}\Pr \left[  R_i(l)-  \mu_i 2^{k_l} - \sqrt{L\log t_n}
 (\sqrt{2} - 1) \sqrt{2^{k_l}}  \geq 0 \right] \\
\label{eqn:34} & = &  \sum_{l=1}^{K} \Pr \left[ R_i(l) - \mu_i
2^{k_l}  \geq  \sqrt{L\log t_n}
 (\sqrt{2} - 1) \sqrt{2^{k_l}}  \right] \\ \label{eqn:35}
 &  = & \sum_{l=1}^{K}  \Pr \left[ \sum_{s\in\mathcal{S}_i}(s O^s_i(l)   -  s 2^{k_l-1}
\pi_s^i )  \geq \sqrt{L\log t_n}
 (\sqrt{2} - 1) \sqrt{2^{k_l}} \right],
\end{eqnarray}
where  $O^s_i(l)$ denote the number of occurrences of state~$s$ on
arm~$i$ in the $l$th segment. The following Chernoff Bound will be
used to bound (\ref{eqn:35}).

{\em Lemma $2$} (Chernoff Bound, Theorem $2.1$ in
\cite{Gillman:98}): Consider a finite state, irreducible, aperiodic
and reversible Markov chain with state space $\mathcal{S}$,
transition probabilities $P$, and an initial distribution
$\mathbf{q}$. Let $N_\mathbf{q} = \lvert (\frac{q_x}{\pi_x}),x \in
\mathcal{S}\rvert_2$. Let $\epsilon$ be the eigenvalue gap given by
$1- \lambda_2$, where $\lambda_2$ is the second largest eigenvalue
of the matrix $P$. Let $A \subset \mathcal{S}$ and $T_A(t)$ be the
number of times that states in $A$ are visited up to time $t$. Then
for any $\gamma \geq 0$, we have\begin{eqnarray} \Pr (T_A(t) - t
\pi_A \geq \gamma) \leq (1+\frac{\gamma \epsilon}{10 t})
N_\mathbf{q} e^{-\gamma^2\epsilon /20 t}.
\end{eqnarray}

Using Lemma~$2$, we have
\begin{eqnarray}
 & & \Pr \left[ \sum_{s\in\mathcal{S}_i}(s O^s_i(l)   -  s 2^{k_l-1}
\pi_s^i )  \geq \sqrt{L\log t_n}
 (\sqrt{2} - 1) \sqrt{2^{k_l}} \right] \\
  & = & \Pr \left[ \sum_{s\in\mathcal{S}_i}(s O^s_i(l)   -  s 2^{k_l-1}
\pi_s^i )  \geq \sqrt{L\log t_n}
 (\sqrt{2} - 1) \sqrt{2^{k_l}} \left( \frac{\sum_{\mathcal{S}_i} s }{\sum_{\mathcal{S}_i}s}\right)  \right] \\
&  = &  \Pr \left[ \sum_{s\in\mathcal{S}_i} \left(s O^s_i(l)   -  s
2^{k_l-1} \pi_s^i -  \sqrt{L\log t_n}
 (\sqrt{2} - 1) \sqrt{2^{k_l}} \left( \frac{s }{\sum_{\mathcal{S}_i}s} \right) \right) \geq  0 \right] \\
&  = &  \Pr \left[ \sum_{s\in\mathcal{S}_i, s \neq 0} \left(s
O^s_i(l)   -  s 2^{k_l-1} \pi_s^i -  \sqrt{L\log t_n}
 (\sqrt{2} - 1) \sqrt{2^{k_l}} \left( \frac{s }{\sum_{\mathcal{S}_i}s} \right) \right) \geq  0 \right] \\
  &  \leq &  \sum_{s\in\mathcal{S}_i, s \neq 0 }  \Pr \left[ s O^s_i(l)   -  s
2^{k_l-1} \pi_s^i -  \sqrt{L\log t_n}
 (\sqrt{2} - 1) \sqrt{2^{k_l}} \left( \frac{s }{\sum_{\mathcal{S}_i}s} \right) \geq  0 \right] \\
 &  \leq &  \sum_{s\in\mathcal{S}_i, s \neq 0 }  \Pr \left[  O^s_i(l)   -
2^{k_l-1} \pi_s^i \geq     \sqrt{L\log t_n}
 (\sqrt{2} - 1) \sqrt{2^{k_l}} \left( \frac{1 }{\sum_{\mathcal{S}_i}s} \right) \right] \\
  &  \leq &  \sum_{s\in\mathcal{S}_i}  \Pr \left[  O^s_i(l)   -
2^{k_l-1} \pi_s^i \geq     \sqrt{L\log t_n}
 (\sqrt{2} - 1) \sqrt{2^{k_l}} \left( \frac{1 }{\sum_{\mathcal{S}_i}s} \right) \right] \\
 &  =  &   |\mathcal{S}_i| \left(1+ \frac{(\sqrt{2} - 1)\epsilon_i \sqrt{L
\log t_n }}{10 \sum_{s \in {S}_i} s} \frac{1}{\sqrt{2^{k_l}}}
\right) N_{\mathbf{q}^i}
t_n^{-((3-2\sqrt{2})L\epsilon_i)/(20(\sum_{s \in {S}_i} s)^2))} .
\end{eqnarray}
Thus we have
\begin{eqnarray}
& & Pr \left[ w_i \bar{s}_{i}(t_n)\geq w_i\mu_i + \sqrt{L w_i \log
t_n} \right] \\
 & \leq & K |\mathcal{S}_i|  N_{\mathbf{q}^i}
t_n^{-((3-2\sqrt{2})L\epsilon^i)/(20(\sum_{s \in {S}_i} s)^2))} \nonumber \\
 && + |\mathcal{S}_i| \frac{\sqrt{2}\epsilon_i \sqrt{L \log
t_n }}{10 \sum_{s \in {S}_i} s} N_{\mathbf{q}^i}
t_n^{-(3-2\sqrt{2})(L\epsilon^i/(20(\sum_{s \in {S}_i} s)^2))} \\
& = &    \left(  K +       \frac{\sqrt{2}\epsilon_i \sqrt{L \log t_n
}}{10 \sum_{s \in {S}_i} s}             \right) |\mathcal{S}_i|
N_{\mathbf{q}^i} t_n^{-(3-2\sqrt{2})(L\epsilon^i/(20(\sum_{s \in
{S}_i} s)^2))} \\ \label{eqn:44} & \leq  &    \left(  \frac{\log
t_n}{\log 2} + \frac{\sqrt{2}\epsilon_i \sqrt{L \log t_n }}{10
\sum_{s \in {S}_i} s} \right) |\mathcal{S}_i| N_{\mathbf{q}^i}
t_n^{-(3-2\sqrt{2})(L\epsilon^i/(20(\sum_{s \in {S}_i} s)^2))} \\
& \leq  &    \left(  \frac{1}{\log 2} + \frac{\sqrt{2}\epsilon_i
\sqrt{L }}{10 \sum_{s \in {S}_i} s} \right) |\mathcal{S}_i|
N_{\mathbf{q}^i} t_n^{1/2-(3-2\sqrt{2})(L\epsilon^i/(20(\sum_{s \in
{S}_i} s)^2))},
\end{eqnarray}
where (\ref{eqn:44}) follows from the fact $ K \leq \log_2 t_n$.
Since $L \geq  \frac{30 r_{\max}^2 }{( 3 - 2\sqrt{2})\epsilon_i}$ ,
 we arrive at
\begin{eqnarray}
\Pr [\bar{s}_{i}(t_n)\geq \mu_i + C_{t_n,w_i}] \leq \left(
\frac{1}{\log 2} + \frac{\sqrt{2}\epsilon_i \sqrt{L }}{10 \sum_{s
\in {S}_i} s} \right) |\mathcal{S}_i| N_{\mathbf{q}^i} t_n^{-1}.
\end{eqnarray}
Similarly, it can be shown that
 \begin{eqnarray}
\Pr [ \bar{s}_j(t_n) \leq \mu_j - C_{t_n,w_j}] \leq \left(
\frac{1}{\log 2} + \frac{\sqrt{2}\epsilon_j \sqrt{L }}{10 \sum_{s
\in {S}_j} s} \right) |\mathcal{S}_j| N_{\mathbf{q}^i} t_n^{-1}.
\end{eqnarray}
Thus
 \begin{eqnarray}\label{eqn:upperboundij}
\Pr[i,j,n] \leq \left[ \left( \frac{1}{\log 2} +
\frac{\sqrt{2}\epsilon_j \sqrt{L }}{10 \sum_{s \in {S}_j} s} \right)
|\mathcal{S}_j|  + \left( \frac{1}{\log 2} +
\frac{\sqrt{2}\epsilon_i \sqrt{L }}{10 \sum_{s \in {S}_i} s} \right)
|\mathcal{S}_i|  \right] N_{\mathbf{q}^i}  t_n ^{-1}.
\end{eqnarray}
Thus the regret caused by engaging bad arms in the $n$th
exploitation epoch is bounded by
\begin{eqnarray}\label{eqn:51}
4^{n-1}2  \sum_{j= 2}^N( \mu_{\sigma(1)} - \mu_{\sigma(j)} ) \left[
\left( \frac{1}{\log 2} + \frac{\sqrt{2}\epsilon_j \sqrt{L }}{10
\sum_{s \in {S}_j} s} \right) |\mathcal{S}_j|  + \left(
\frac{1}{\log 2} + \frac{\sqrt{2}\epsilon_1 \sqrt{L }}{10 \sum_{s
\in {S}_1} s} \right) |\mathcal{S}_1|  \right] \frac{1}{ \pi_{\min}}
t_n ^{-1}.
\end{eqnarray}
By (\ref{eqn:ubexploitation}) and $t_n \geq \frac{2}{3} 4^{n-1}$,
the bound in (\ref{eqn:51}) becomes
\begin{eqnarray}\label{eqn:regretexploitation}
3\lceil \log_4 (\frac{3}{2} (t- N) +1) \rceil   \frac{1}{
\pi_{\min}}\sum_{j= 2}^N( \mu_{\sigma(1)} - \mu_{\sigma(j)} )\sum_{
k = 1,j}  \left( \frac{1}{\log 2} + \frac{\sqrt{2}\epsilon_k \sqrt{L
}}{10 \sum_{s \in {S}_k} s} \right) |\mathcal{S}_k|  .
\end{eqnarray}
Combining~(\ref{eqn:regreform})~(\ref{eqn:firstterm})~(\ref{eqn:MPRploration})~(\ref{eqn:regretexploitation}),
we arrive at the upper bound of regret given in
(\ref{eqn:totalregretsingle}).

We point out that the same Chernoff bound given in Lemma~$2$ is also
used in~\cite{Tekin:10} to handle the \emph{rested} Markovian reward
MAB problem. Note that the Cheroff bound in~\cite{Gillman:98}
requires that all the observations used in calculating the sample
means ($\bar{s}_{i}$ and $\bar{s}_{j}$ in (\ref{eqn:twoevents})) are
from a continuously evolving Markov process. This condition is
naturally satisfied in the rested MAB problem. However, for the
restless MAB problem considered here, the sample means are
calculated using observations from multiple epochs, which are
noncontiguous segments of the Markovian sample path. As detailed in
the above proof, the desired bound on the probabilities of the
events in (\ref{eqn:twoevents}) is ensured by the carefully chosen
(growing) lengths of the exploration and exploitation epochs.

\section*{Appendix B. Proof of Theorem~\ref{thm:singleplayerarbitraryclose} and Theorem~\ref{thm:singleplayerarbitrarycloseM}}\label{sec:AppendixB}
Recall in Theorem~$1$ and Theorem~$3$, $L$ and $D$ are fixed \emph{a
priori}. Now we choose $L(t) \to \infty$ as $t \to \infty$ and
$\frac{D(t)}{L(t)} \to \infty$ as $t \to \infty$. By the same
reasoning in the proof of Theorem~$\ref{thm:singleuser}$, the regret
has three parts: The regret caused by arm switching, the regret
caused by playing bad arms in the exploration epochs, and the regret
caused by playing bad arms in the exploitation epochs. It will be
shown that each part part of the regret is on a lower order or on
the same order of $f(t)\log t$.

The number of arm switchings is upper bounded by $N \log_2(t/N+1)$.
So the regret caused by arm switching is upper bounded by
\begin{eqnarray}
N  \log_2(t/N+1) A_{\max}.
\end{eqnarray}
Since $f(t) \to \infty$ as $t \to \infty$, we have
\begin{eqnarray}
\lim_{t \to \infty} \frac{N  \log_2(t/N+1)  \max_i A_i}{f(t)\log t}
= 0.
\end{eqnarray}
Thus the regret caused by arm switching is on a lower order than
$f(t)\log t$.

 The regret caused by playing bad arms in the exploration
epochs is bounded by
\begin{eqnarray}
\frac{1}{3}[4(3D(t) \log t +1)-1]   \left( \lceil \frac{N}{M} \rceil
\sum_{i = 1}^M \mu_{\sigma(i)} -
\sum_{i=1}^{N}\mu_{\sigma(i)}\right).
\end{eqnarray}
Thus the regret caused by playing bad arms in the exploration epochs
is on the same order of $f(t)\log t$.

For the regret caused by playing bad arms in the exploitation
epochs, it is shown below that the time spent on a bad arm $i$ can
be bounded by a constant independent of $t$. Since
$\frac{D(t)}{L(t)} \to \infty $ as $t \to \infty$, there exists a
time $t_1$ such that $\forall t \geq t_1$, $D(t) \geq
\frac{4L(t)}{(\mu_{\sigma(1)} - \mu_{\sigma(2)})^2}$. There also
exists a time $t_2$ such that $\forall t \geq t_2$, $L(t) \geq
 \frac{70 r_{\max}^2 }{( 3 - 2\sqrt{2})\epsilon_{\min}}
$. The time spent on playing bad arms before $t_3 = \max(t_1,t_2)$
is at most $t_3$, and the caused regret is at most $( \sum_{i=
1}^M\mu_{\sigma(i)})t_3$. After $t_3$, the time spent on each bad
arm $i$ is upper bounded by (following similar reasoning from
(\ref{eqn:twothreeevents}) to (\ref{eqn:regretexploitation}))
\begin{eqnarray}
\frac{\pi^2}{2}
\frac{|\mathcal{S}_i|+|\mathcal{S}_{\sigma(1)}|}{\pi_{\min}}(1+
\frac{\epsilon_{\max} \sqrt{L(t_5)}}{10s_{\min}}).
\end{eqnarray}
An upper bound for the corresponding regret is
\begin{eqnarray}
\frac{\pi^2}{2} \sum_{j = 1}^M\sum_{j = M+1}^N( \mu_{\sigma(j)} -
\mu_{\sigma(i)} )\sum_{ k = 1,j}  \left( \frac{1}{\log 2} +
\frac{\sqrt{2}\epsilon_k \sqrt{L }}{10 \sum_{s \in {S}_k} s} \right)
|\mathcal{S}_k| \frac{1}{ \pi_{\min}},
\end{eqnarray}
which is a constant independent of time $t$. Thus the regret caused
by playing bad arms in the exploration epochs is on a lower order
than $f(t)\log t$.

Because each part of the regret is on a lower order than or on the
same order of $f(t)\log t $, the total regret is on the same order
of $f(t)\log t$.

\section*{Appendix C. Proof of Theorem~\ref{thm:singleuserM}}

For the case of $M>1$, we first rewrite the definition of regret as
\begin{eqnarray}
r_{\Phi}(t) & = & t \sum_{i = 1}^M \mu_{\sigma(i)}  -
\mbbE_{\Phi}R(t)
\\ & = & \sum_{i=1}^{N} [\mu_i \mbbE[T_i(t)] -
\mbbE[\sum_{n=1}^{T_i(t)}s_i(t_i(n))]] + \left[  t \sum_{i = 1}^M
\mu_{\sigma(i)} - \sum_{i=1}^{N} \mu_i \mbbE[T_i(t)]\right].
\label{eqn:regreformM}
\end{eqnarray}
To show that the regret has a logarithmic order, it is sufficient to
show that the two terms in (\ref{eqn:regreformM})
 have logarithmic orders. The first term in (\ref{eqn:regreformM}) can be considered as the
regret caused by transient effect. The second term can be considered
as the regret caused by engaging a bad arm. Similar to what we have
done for $M=1$, we bound the regret caused by transient effect based
on Lemma $1$ and upper bounds on numbera of epochs in
(\ref{eqn:ubexploration}) and (\ref{eqn:ubexploitation}). An
logarithmic upper bound of the first term in (\ref{eqn:regreformM})
is
\begin{eqnarray}\label{eqn:firsttermM}
\sum_{i=1}^{N} [\mu_i \mbbE[T_i(t)] -
\mbbE[\sum_{n=1}^{T_i(t)}s_i(t_i(n))]] \leq (M\lceil \log_4
(\frac{3}{2} (t- N) +1) \rceil
 +N(\lfloor \log _4 (3D \log t +1) \rfloor + 1) )A_{\max}.
\end{eqnarray}

Next we show that the second term of (\ref{eqn:regreformM}) has a
logarithmic order by bounding the total time spent on the bad arms.
We first bound the time spent on the bad arms during the exploration
epochs. By (\ref{eqn:exploration}) the regret caused by playing bad
arms in the exploration epochs is
\begin{eqnarray}\label{eqn:MPRplorationM}
\frac{1}{3}[4(3D \log t +1)-1]   \left( \lceil \frac{N}{M} \rceil
\sum_{i = 1}^M \mu_{\sigma(i)} -
\sum_{i=1}^{N}\mu_{\sigma(i)}\right).
\end{eqnarray}

Next, we bound the time spent on the bad arms during the
exploitation epochs. By (\ref{eqn:upperboundij}) the regret caused
by engaging bad arms in the $n$th exploitation epoch is bounded by
\begin{eqnarray}\label{eqn:51}
4^{n-1}2  \sum_{i = 1}^M \sum_{j= 2}^N( \mu_{\sigma(i)} -
\mu_{\sigma(j)} ) \left[ \left( \frac{1}{\log 2} +
\frac{\sqrt{2}\epsilon_j \sqrt{L }}{10 \sum_{s \in {S}_j} s} \right)
|\mathcal{S}_j|  + \left( \frac{1}{\log 2} +
\frac{\sqrt{2}\epsilon_1 \sqrt{L }}{10 \sum_{s \in {S}_1} s} \right)
|\mathcal{S}_1|  \right] \frac{1}{ \pi_{\min}} t_n ^{-1}.
\end{eqnarray}
By (\ref{eqn:ubexploitation}) and $t_n \geq \frac{2}{3} 4^{n-1}$,
the bound in (\ref{eqn:51}) becomes
\begin{eqnarray}\label{eqn:regretexploitationM}
3\lceil \log_4 (\frac{3}{2} (t- N) +1) \rceil   \frac{1}{
\pi_{\min}}\sum_{j = 1}^M \sum_{i= M+1}^N( \mu_{\sigma(j)} -
\mu_{\sigma(i)} )\sum_{ k = j,i}  \left( \frac{1}{\log 2} +
\frac{\sqrt{2}\epsilon_k \sqrt{L }}{10 \sum_{s \in {S}_k} s} \right)
|\mathcal{S}_k|  .
\end{eqnarray}
Combining~(\ref{eqn:regreformM})~(\ref{eqn:firsttermM})~(\ref{eqn:MPRplorationM})~(\ref{eqn:regretexploitationM}),
we arrive at the upper bound of regret given in
(\ref{eqn:totalregretsingleM}).

\section*{Appendix D. Proof of Theorem~$5$}
We first rewrite the definition of regret as
\begin{eqnarray}\label{eqn:MPregreform}
r_{\Phi}(t) & =  & t\sum_{i=1}^{M}\mu_{\sigma(i)}  -
\mbbE_{\Phi}R(t) \\   & = &\sum_{i=1}^{N} [\mu_i \mbbE[T_i(t)] -
\mbbE[\sum_{n=1}^{T_i(t)}s_i(t_i(n))]] +\left[
t\sum_{i=1}^{M}\mu_{\sigma(i)}- \sum_{i=1}^{N} \mu_i \mbbE[T_i(t)]
\right].\label{eqn:regreform1}
\end{eqnarray}

Using Lemma~$1$ the first term in (\ref{eqn:regreform1}) can be
bounded by the following constant under the endogenous restless
model (it is zero under the exogenous model):
\begin{eqnarray}\label{eqn:endo12}
(M\lceil \log_4 (\frac{3t}{2M} +1) \rceil
 +  N(\lfloor \log _4 (3D \log t +1) \rfloor + 1) )M A_{\max},
\end{eqnarray}
which has a logarithmic order.

 We are going to show that the second
term in (\ref{eqn:MPregreform}) has a logarithmic order. It will be
verified by bounding regret in both exploitation and exploration
epochs by logarithmic order.

The upper bound on $T_{O}(t)$ in (\ref{eqn:exploration}) still holds
and consequently the regret caused by engaging bad arms in the
exploration epochs by time $t$ is upper bounded by
\begin{eqnarray}\label{eqn:MPRploration1}
\frac{1}{3}[4(3D \log t +1)-1]  \left(N
\sum_{i=1}^{M}\mu_{\sigma(i)} - M
\sum_{i=1}^{N}\mu_{\sigma(i)}\right).
\end{eqnarray}

The second reason for regret in the second term of
(\ref{eqn:MPregreform}) is not playing the expected arms in the
exploitation epochs. If in the $m$th subepoch player $k$ plays the
$(m+k) \oslash M$ best arm, then every time the best $M$ arms are
played and there is no conflict. But arm $a^*_{(m+k) \oslash M}$ may
not be the $(m+k) \oslash M$ best arm. Bounding the probabilities of
mistakes can lead to an upper bound on the regret caused in the
exploitation epochs.

We adopt the same notations in Appendix~A. The upper bound on
$\Pr[i,j,n]$ in (\ref{eqn:upperboundij}) still holds. Since
different subepochs in the exploitation epochs are symmetric, the
expected regret in different subepochs are the same. In the first
subepoch, player $k$ aims at arm $\sigma(k)$. In the model where no
player in conflict gets any reward, player $k$ failing to identify
arm $\sigma(k)$ in the first subepoch of the $n$th exploitation
epoch can lead to a regret no more than $\sum_{m = 1}^M 2\mu_{m}
\times 4^{n-1} $. In the model where players share the reward,
player $k$ failing to identify arm $\sigma(k)$ in the first subepoch
of the $n$th exploitation epoch can lead to a regret no more than
$2\mu_{k} \times 4^{n-1} $. Thus an upper bound for regret in the
$n$th exploitation epoch for no reward conflict model can be
obtained as
\begin{eqnarray}
4^{n-1}2 M t_n ^{-1} \frac{1}{ \pi_{\min}} \sum_{m = 1}^M \mu_{m}
\sum_{j = 1}^{M} \sum_{i = 1, i \neq j}^{N} \sum_{ k = i,j} \left(
\frac{1}{\log 2} + \frac{\sqrt{2}\epsilon_k \sqrt{L }}{10 \sum_{s
\in {S}_k} s} \right) |\mathcal{S}_k| .
\end{eqnarray}
and an upper bound for regret in the $n$th exploitation epoch for
sharing reward model can be obtained as
\begin{eqnarray}
4^{n-1}2 M t_n ^{-1} \frac{1}{ \pi_{\min}} \sum_{j = 1}^{M}
\mu_{\sigma(j)} \sum_{i = 1, i \neq j}^{N} \sum_{ k = i,j} \left(
\frac{1}{\log 2} + \frac{\sqrt{2}\epsilon_k \sqrt{L }}{10 \sum_{s
\in {S}_k} s} \right) |\mathcal{S}_k| .
\end{eqnarray}
By time $t$, we have
\begin{eqnarray}\label{eqn:ubexploitation1}
n_I(t) \leq \lceil \log_4 (\frac{3t}{2M} +1) \rceil.
\end{eqnarray}
From the upper bound on the number of the exploitation epochs given
in (\ref{eqn:ubexploitation}), and also the fact that $t_n \geq
\frac{2}{3} 4^{n-1}$, we have the following upper bound on regret
caused in the exploitation epochs under no reward conflict model by
time $t$ (Denoted by $r_{\Phi,I}(t)$):
\begin{eqnarray} \label{eqn:MPUPI}
r_{\Phi,I}(t) &\leq &3 M \lceil \log_4 (\frac{3t}{2M} +1) \rceil
\sum_{m = 1}^M \mu_{m} \frac{1}{ \pi_{\min}}  \sum_{j = 1}^{M}
\sum_{i = 1, i \neq j}^{N} \sum_{ k = i,j} \left( \frac{1}{\log 2} +
\frac{\sqrt{2}\epsilon_k \sqrt{L }}{10 \sum_{s \in {S}_k} s} \right)
|\mathcal{S}_k| .
\end{eqnarray}
and the upper bound under sharing reward conflict model
\begin{eqnarray} \label{eqn:MPUPI1}
r_{\Phi,I}(t) &\leq &3 M \lceil \log_4 (\frac{3t}{2M} +1) \rceil
 \frac{1}{ \pi_{\min}}  \sum_{j = 1}^{M} \mu_{\sigma(j)}
\sum_{i = 1, i \neq j}^{N} \sum_{ k = i,j} \left( \frac{1}{\log 2} +
\frac{\sqrt{2}\epsilon_k \sqrt{L }}{10 \sum_{s \in {S}_k} s} \right)
|\mathcal{S}_k| .
\end{eqnarray}

Combining~(\ref{eqn:MPregreform})~(\ref{eqn:endo12})~(\ref{eqn:MPRploration})~(\ref{eqn:MPUPI})~(\ref{eqn:MPUPI1}),
we arrive at the upper bounds of regret given in (\ref{eqn:2}).

\section*{Appendix E. Proof of Theorem~\ref{thm:multipleplayersarbitraryclose}}

We set  $L(t) \to \infty$ as $t \to \infty$ and $\frac{D(t)}{L(t)}
\to \infty$ as $t \to \infty$. The regret has three parts: the
transient effect of arms, the regret caused by playing bad arms in
the exploration epochs, and the regret caused by mistakes in the
exploitation epochs. It will be shown that each part of the regret
is on a lower order or at the same order of $f(t)\log t$. The
transient effect of arms is the same as in Theorem~$3$. Thus it is
upper bounded by a constant under the exogenous restless model and
on the order of $\log t$ under the endogenous restless model. Thus
it is on a lower order than $f(t)\log t$

The regret caused by playing bad arms in the exploration epochs is
bounded by
\begin{eqnarray}\label{eqn:1}
\frac{1}{3}[4(3D(t) \log t +1)-1]  \left( N
\sum_{i=1}^{M}\mu_{\sigma(i)} - M
\sum_{i=1}^{N}\mu_{\sigma(i)}\right).
\end{eqnarray}
Since $D(t) = f(t)$, regret in (\ref{eqn:1}) is on the same order
$f(t)\log t$.

For the regret caused by playing bad arms in the exploitation
epochs, it is shown below that the time spent on a bad arm $i$ can
be bounded by a constant independent of $t$.

Since $\frac{D(t)}{L(t)} \to \infty $ as $t \to \infty$, there
exists a time $t_1$ such that $\forall t \geq t_1$, $D(t) \geq
\frac{4L(t)}{(\min_{j\leq M}( \mu_{\sigma(j)} -
\mu_{\sigma(j+1)}))^2}$. There also exists a time $t_2$ such that
$\forall t \geq t_2$, $L(t)  \geq
 \frac{70 r_{\max}^2 }{( 3 - 2\sqrt{2})\epsilon_{\min}}
$. The time spent on playing bad arms before $t_3 = \max(t_1,t_2)$
is at most $t_3$, and the time spent on playing bad arms after $t_3$
is also bounded by a finite constant, which can be found in a
similar manner in Appendix B. Thus the regret caused by mistakes in
the exploitation epochs is on a lower order than $f(t)\log t $.

Because each part of the regret is on a lower order than or on the
same order of $f(t)\log t$, the total regret is on the same order of
$f(t)\log t$.

\section*{Appendix F. Proof of Theorem~\ref{thm:5}}
At each time, regret incurs if one of the following three events
happens: (i) at least one player incorrectly identifies the set of
$M$ best arms in the exploitation sequence, (ii) at least one player
is exploring, (iii) at least a collision occurs among the players.
In the following, we will bound the expected number of the
occurrences of these three events by the logarithmic order with
time.

We first consider events (i) and (ii). Define a singular slot as the
time slot in which either (i) or (ii) occurs. Based on the previous
theorems, the local expected number of learning mistakes at each
players is bounded by the logarithmic order with time. Furthermore,
the cardinality of the local exploration sequence at each player is
also bounded by the logarithmic order with time. We thus have that
the expected number of singular slots is bounded by the logarithmic
order with time, \ie the expected number of the occurrences of
events (i) and (ii) is bounded by the logarithmic order with time.

To prove the theorem, it remains to show that the expected number of
collisions in all non-singular slots is also bounded by the
logarithmic order with time. Consider the contiguous period
consisting of all slots between two successive singular slots.
During this period, all players correctly identify the $M$ best arms
and a collision occurs if and only if at least two players choose
the same arm. Due to the randomized arm selection after each
collision, it is clear that, in this period, the expected number of
collisions before all players are orthogonalized into the $M$ best
arms is bounded by a constant uniform over time. Since the expected
number of such periods has the same order as the expected number of
singular slots, the expected number of such periods is bounded by
the logarithmic order with time. The expected number of collisions
over all such periods is thus bounded by the logarithmic order with
time, \ie the expected number of collisions in all non-singular
slots is bounded by the logarithmic order with time. We thus proved
the theorem.

\section*{Appendix G. Transition Matrix for Simulation in Sec.~\ref{sec:simulation}}

The transition matrix for arm $1$ is
\begin{eqnarray*} {\tiny \left[  \begin{matrix}0.0401  \text{ }    0.0787  \text{ }    0.0188  \text{ }    0.0572  \text{ }    0.0531  \text{ }    0.0569  \text{ }    0.0491  \text{ }    0.0145  \text{ }    0.0583  \text{ }    0.0244  \text{ }    0.0195  \text{ }    0.0694  \text{ }    0.0654  \text{ }    0.0256  \text{ }    0.0656  \text{ }    0.0707  \text{ }    0.0809  \text{ }    0.0283  \text{ }    0.0322  \text{ }    0.0914  \\
0.0787  \text{ }    0.0448  \text{ }    0.0677  \text{ }    0.0165  \text{ }    0.0674  \text{ }    0.0545  \text{ }    0.0537  \text{ }    0.0622  \text{ }    0.0653  \text{ }    0.0491  \text{ }    0.0163  \text{ }    0.0613  \text{ }    0.0679  \text{ }    0.0580  \text{ }    0.0216  \text{ }    0.0580  \text{ }    0.0042  \text{ }    0.0219  \text{ }    0.0650  \text{ }    0.0659  \\
0.0188  \text{ }    0.0677  \text{ }    0.0885  \text{ }    0.0107  \text{ }    0.0518  \text{ }    0.0687  \text{ }    0.0243  \text{ }    0.0997  \text{ }    0.0562  \text{ }    0.0663  \text{ }    0.0674  \text{ }    0.0005  \text{ }    0.1048  \text{ }    0.0571  \text{ }    0.0562  \text{ }    0.0411  \text{ }    0.0125  \text{ }    0.0308  \text{ }    0.0593  \text{ }    0.0176  \\
0.0572  \text{ }    0.0165  \text{ }    0.0107  \text{ }    0.0083  \text{ }    0.0520  \text{ }    0.0802  \text{ }    0.0310  \text{ }    0.0731  \text{ }    0.0967  \text{ }    0.0697  \text{ }    0.0773  \text{ }    0.0630  \text{ }    0.0222  \text{ }    0.0229  \text{ }    0.0910  \text{ }    0.0036  \text{ }    0.0925  \text{ }    0.0180  \text{ }    0.0049  \text{ }    0.1091  \\
0.0531  \text{ }    0.0674  \text{ }    0.0518  \text{ }    0.0520  \text{ }    0.0025  \text{ }    0.0801  \text{ }    0.0935  \text{ }    0.0495  \text{ }    0.0076  \text{ }    0.0097  \text{ }    0.0318  \text{ }    0.1150  \text{ }    0.1095  \text{ }    0.0355  \text{ }    0.0664  \text{ }    0.0160  \text{ }    0.0449  \text{ }    0.0321  \text{ }    0.0748  \text{ }    0.0069  \\
0.0569  \text{ }    0.0545  \text{ }    0.0687  \text{ }    0.0802  \text{ }    0.0801  \text{ }    0.0551  \text{ }    0.0741  \text{ }    0.0426  \text{ }    0.0085  \text{ }    0.0405  \text{ }    0.0642  \text{ }    0.0234  \text{ }    0.0055  \text{ }    0.0196  \text{ }    0.0744  \text{ }    0.0095  \text{ }    0.0804  \text{ }    0.0468  \text{ }    0.0559  \text{ }    0.0590  \\
0.0491  \text{ }    0.0537  \text{ }    0.0243  \text{ }    0.0310  \text{ }    0.0935  \text{ }    0.0741  \text{ }    0.0003  \text{ }    0.0852  \text{ }    0.0199  \text{ }    0.0733  \text{ }    0.1019  \text{ }    0.0055  \text{ }    0.0491  \text{ }    0.0898  \text{ }    0.0580  \text{ }    0.1040  \text{ }    0.0531  \text{ }    0.0114  \text{ }    0.0193  \text{ }    0.0036  \\
0.0145  \text{ }    0.0622  \text{ }    0.0997  \text{ }    0.0731  \text{ }    0.0495  \text{ }    0.0426  \text{ }    0.0852  \text{ }    0.0028  \text{ }    0.0722  \text{ }    0.0563  \text{ }    0.0690  \text{ }    0.0234  \text{ }    0.0237  \text{ }    0.0496  \text{ }    0.0626  \text{ }    0.0737  \text{ }    0.0442  \text{ }    0.0496  \text{ }    0.0344  \text{ }    0.0117  \\
0.0583  \text{ }    0.0653  \text{ }    0.0562  \text{ }    0.0967  \text{ }    0.0076  \text{ }    0.0085  \text{ }    0.0199  \text{ }    0.0722  \text{ }    0.0189  \text{ }    0.0649  \text{ }    0.1013  \text{ }    0.0736  \text{ }    0.0413  \text{ }    0.0371  \text{ }    0.0376  \text{ }    0.0961  \text{ }    0.0072  \text{ }    0.0672  \text{ }    0.0507  \text{ }    0.0195  \\
0.0244  \text{ }    0.0491  \text{ }    0.0663  \text{ }    0.0697  \text{ }    0.0097  \text{ }    0.0405  \text{ }    0.0733  \text{ }    0.0563  \text{ }    0.0649  \text{ }    0.1114  \text{ }    0.0887  \text{ }    0.0175  \text{ }    0.0031  \text{ }    0.0690  \text{ }    0.0260  \text{ }    0.0090  \text{ }    0.0596  \text{ }    0.1015  \text{ }    0.0405  \text{ }    0.0195  \\
0.0195  \text{ }    0.0163  \text{ }    0.0674  \text{ }    0.0773  \text{ }    0.0318  \text{ }    0.0642  \text{ }    0.1019  \text{ }    0.0690  \text{ }    0.1013  \text{ }    0.0887  \text{ }    0.0747  \text{ }    0.0866  \text{ }    0.0428  \text{ }    0.0089  \text{ }    0.0152  \text{ }    0.0428  \text{ }    0.0287  \text{ }    0.0337  \text{ }    0.0083  \text{ }    0.0212  \\
0.0694  \text{ }    0.0613  \text{ }    0.0005  \text{ }    0.0630  \text{ }    0.1150  \text{ }    0.0234  \text{ }    0.0055  \text{ }    0.0234  \text{ }    0.0736  \text{ }    0.0175  \text{ }    0.0866  \text{ }    0.0961  \text{ }    0.0235  \text{ }    0.0617  \text{ }    0.0261  \text{ }    0.1233  \text{ }    0.0238  \text{ }    0.0417  \text{ }    0.0177  \text{ }    0.0469  \\
0.0654  \text{ }    0.0679  \text{ }    0.1048  \text{ }    0.0222  \text{ }    0.1095  \text{ }    0.0055  \text{ }    0.0491  \text{ }    0.0237  \text{ }    0.0413  \text{ }    0.0031  \text{ }    0.0428  \text{ }    0.0235  \text{ }    0.0611  \text{ }    0.0354  \text{ }    0.0705  \text{ }    0.0817  \text{ }    0.0815  \text{ }    0.0221  \text{ }    0.0590  \text{ }    0.0298  \\
0.0256  \text{ }    0.0580  \text{ }    0.0571  \text{ }    0.0229  \text{ }    0.0355  \text{ }    0.0196  \text{ }    0.0898  \text{ }    0.0496  \text{ }    0.0371  \text{ }    0.0690  \text{ }    0.0089  \text{ }    0.0617  \text{ }    0.0354  \text{ }    0.0150  \text{ }    0.1057  \text{ }    0.0951  \text{ }    0.0401  \text{ }    0.1038  \text{ }    0.0466  \text{ }    0.0235  \\
0.0656  \text{ }    0.0216  \text{ }    0.0562  \text{ }    0.0910  \text{ }    0.0664  \text{ }    0.0744  \text{ }    0.0580  \text{ }    0.0626  \text{ }    0.0376  \text{ }    0.0260  \text{ }    0.0152  \text{ }    0.0261  \text{ }    0.0705  \text{ }    0.1057  \text{ }    0.0657  \text{ }    0.0436  \text{ }    0.0088  \text{ }    0.0665  \text{ }    0.0035  \text{ }    0.0350  \\
0.0707  \text{ }    0.0580  \text{ }    0.0411  \text{ }    0.0036  \text{ }    0.0160  \text{ }    0.0095  \text{ }    0.1040  \text{ }    0.0737  \text{ }    0.0961  \text{ }    0.0090  \text{ }    0.0428  \text{ }    0.1233  \text{ }    0.0817  \text{ }    0.0951  \text{ }    0.0436  \text{ }    0.0289  \text{ }    0.0080  \text{ }    0.0368  \text{ }    0.0400  \text{ }    0.0182  \\
0.0809  \text{ }    0.0042  \text{ }    0.0125  \text{ }    0.0925  \text{ }    0.0449  \text{ }    0.0804  \text{ }    0.0531  \text{ }    0.0442  \text{ }    0.0072  \text{ }    0.0596  \text{ }    0.0287  \text{ }    0.0238  \text{ }    0.0815  \text{ }    0.0401  \text{ }    0.0088  \text{ }    0.0080  \text{ }    0.0110  \text{ }    0.0231  \text{ }    0.2495  \text{ }    0.0459  \\
0.0283  \text{ }    0.0219  \text{ }    0.0308  \text{ }    0.0180  \text{ }    0.0321  \text{ }    0.0468  \text{ }    0.0114  \text{ }    0.0496  \text{ }    0.0672  \text{ }    0.1015  \text{ }    0.0337  \text{ }    0.0417  \text{ }    0.0221  \text{ }    0.1038  \text{ }    0.0665  \text{ }    0.0368  \text{ }    0.0231  \text{ }    0.1075  \text{ }    0.1164  \text{ }    0.0409  \\
0.0322  \text{ }    0.0650  \text{ }    0.0593  \text{ }    0.0049  \text{ }    0.0748  \text{ }    0.0559  \text{ }    0.0193  \text{ }    0.0344  \text{ }    0.0507  \text{ }    0.0405  \text{ }    0.0083  \text{ }    0.0177  \text{ }    0.0590  \text{ }    0.0466  \text{ }    0.0035  \text{ }    0.0400  \text{ }    0.2495  \text{ }    0.1164  \text{ }    0.0103  \text{ }    0.0118  \\
0.0914  \text{ }    0.0659  \text{ }    0.0176  \text{ }    0.1091
\text{ }    0.0069  \text{ }    0.0590  \text{ }    0.0036  \text{ }
0.0117  \text{ }    0.0195  \text{ }    0.0195  \text{ }    0.0212
\text{ }    0.0469  \text{ }    0.0298  \text{ }    0.0235  \text{ }
0.0350  \text{ }    0.0182  \text{ }    0.0459  \text{ }    0.0409
\text{ }    0.0118  \text{ }    0.3227
\end{matrix} \right] }
\end{eqnarray*}

The transition matrix for arm $2$ is
\begin{eqnarray*} {\tiny \left[  \begin{matrix}
0.0266  \text{ }    0.0415  \text{ }    0.0248  \text{ }    0.0896  \text{ }    0.0596  \text{ }    0.0615  \text{ }    0.0847  \text{ }    0.0106  \text{ }    0.0175  \text{ }    0.0734  \text{ }    0.0361  \text{ }    0.0888  \text{ }    0.0906  \text{ }    0.0118  \text{ }    0.0829  \text{ }    0.0442  \text{ }    0.0542  \text{ }    0.0112  \text{ }    0.0439  \text{ }    0.0467  \\
0.0415  \text{ }    0.0544  \text{ }    0.0473  \text{ }    0.0287  \text{ }    0.0852  \text{ }    0.0084  \text{ }    0.0224  \text{ }    0.0515  \text{ }    0.0696  \text{ }    0.0496  \text{ }    0.0397  \text{ }    0.0890  \text{ }    0.0919  \text{ }    0.0244  \text{ }    0.0427  \text{ }    0.0088  \text{ }    0.0808  \text{ }    0.0269  \text{ }    0.0845  \text{ }    0.0529  \\
0.0248  \text{ }    0.0473  \text{ }    0.0164  \text{ }    0.0902  \text{ }    0.0705  \text{ }    0.0828  \text{ }    0.0828  \text{ }    0.0492  \text{ }    0.0346  \text{ }    0.0413  \text{ }    0.0637  \text{ }    0.0804  \text{ }    0.0078  \text{ }    0.0321  \text{ }    0.0798  \text{ }    0.0250  \text{ }    0.0757  \text{ }    0.0692  \text{ }    0.0115  \text{ }    0.0151  \\
0.0896  \text{ }    0.0287  \text{ }    0.0902  \text{ }    0.0768  \text{ }    0.0505  \text{ }    0.0020  \text{ }    0.0499  \text{ }    0.0578  \text{ }    0.0023  \text{ }    0.0545  \text{ }    0.0633  \text{ }    0.0886  \text{ }    0.0395  \text{ }    0.0528  \text{ }    0.0389  \text{ }    0.0927  \text{ }    0.0429  \text{ }    0.0272  \text{ }    0.0365  \text{ }    0.0154  \\
0.0596  \text{ }    0.0852  \text{ }    0.0705  \text{ }    0.0505  \text{ }    0.0840  \text{ }    0.0984  \text{ }    0.0415  \text{ }    0.0660  \text{ }    0.0336  \text{ }    0.0713  \text{ }    0.0017  \text{ }    0.0126  \text{ }    0.0660  \text{ }    0.0249  \text{ }    0.0341  \text{ }    0.0006  \text{ }    0.0903  \text{ }    0.0715  \text{ }    0.0263  \text{ }    0.0113  \\
0.0615  \text{ }    0.0084  \text{ }    0.0828  \text{ }    0.0020  \text{ }    0.0984  \text{ }    0.0812  \text{ }    0.0611  \text{ }    0.0935  \text{ }    0.0379  \text{ }    0.0536  \text{ }    0.0779  \text{ }    0.0592  \text{ }    0.0713  \text{ }    0.0083  \text{ }    0.0052  \text{ }    0.0616  \text{ }    0.0347  \text{ }    0.0617  \text{ }    0.0110  \text{ }    0.0288  \\
0.0847  \text{ }    0.0224  \text{ }    0.0828  \text{ }    0.0499  \text{ }    0.0415  \text{ }    0.0611  \text{ }    0.0638  \text{ }    0.0540  \text{ }    0.0362  \text{ }    0.0761  \text{ }    0.0192  \text{ }    0.0696  \text{ }    0.0350  \text{ }    0.0656  \text{ }    0.0162  \text{ }    0.0684  \text{ }    0.0126  \text{ }    0.0582  \text{ }    0.0690  \text{ }    0.0139  \\
0.0106  \text{ }    0.0515  \text{ }    0.0492  \text{ }    0.0578  \text{ }    0.0660  \text{ }    0.0935  \text{ }    0.0540  \text{ }    0.0470  \text{ }    0.0404  \text{ }    0.0705  \text{ }    0.0865  \text{ }    0.0450  \text{ }    0.0264  \text{ }    0.0516  \text{ }    0.0119  \text{ }    0.0535  \text{ }    0.0694  \text{ }    0.0803  \text{ }    0.0165  \text{ }    0.0185  \\
0.0175  \text{ }    0.0696  \text{ }    0.0346  \text{ }    0.0023  \text{ }    0.0336  \text{ }    0.0379  \text{ }    0.0362  \text{ }    0.0404  \text{ }    0.0758  \text{ }    0.0893  \text{ }    0.0655  \text{ }    0.0721  \text{ }    0.0842  \text{ }    0.0803  \text{ }    0.0595  \text{ }    0.0101  \text{ }    0.0311  \text{ }    0.0158  \text{ }    0.0845  \text{ }    0.0597  \\
0.0734  \text{ }    0.0496  \text{ }    0.0413  \text{ }    0.0545  \text{ }    0.0713  \text{ }    0.0536  \text{ }    0.0761  \text{ }    0.0705  \text{ }    0.0893  \text{ }    0.0579  \text{ }    0.0394  \text{ }    0.0287  \text{ }    0.0564  \text{ }    0.0624  \text{ }    0.0566  \text{ }    0.0089  \text{ }    0.0536  \text{ }    0.0441  \text{ }    0.0050  \text{ }    0.0073  \\
0.0361  \text{ }    0.0397  \text{ }    0.0637  \text{ }    0.0633  \text{ }    0.0017  \text{ }    0.0779  \text{ }    0.0192  \text{ }    0.0865  \text{ }    0.0655  \text{ }    0.0394  \text{ }    0.0548  \text{ }    0.0033  \text{ }    0.0200  \text{ }    0.0820  \text{ }    0.0081  \text{ }    0.1193  \text{ }    0.0826  \text{ }    0.0816  \text{ }    0.0243  \text{ }    0.0311  \\
0.0888  \text{ }    0.0890  \text{ }    0.0804  \text{ }    0.0886  \text{ }    0.0126  \text{ }    0.0592  \text{ }    0.0696  \text{ }    0.0450  \text{ }    0.0721  \text{ }    0.0287  \text{ }    0.0033  \text{ }    0.0602  \text{ }    0.0345  \text{ }    0.0537  \text{ }    0.0197  \text{ }    0.0666  \text{ }    0.0065  \text{ }    0.0160  \text{ }    0.0602  \text{ }    0.0455  \\
0.0906  \text{ }    0.0919  \text{ }    0.0078  \text{ }    0.0395  \text{ }    0.0660  \text{ }    0.0713  \text{ }    0.0350  \text{ }    0.0264  \text{ }    0.0842  \text{ }    0.0564  \text{ }    0.0200  \text{ }    0.0345  \text{ }    0.0449  \text{ }    0.0559  \text{ }    0.0525  \text{ }    0.0589  \text{ }    0.0433  \text{ }    0.0657  \text{ }    0.0115  \text{ }    0.0435  \\
0.0118  \text{ }    0.0244  \text{ }    0.0321  \text{ }    0.0528  \text{ }    0.0249  \text{ }    0.0083  \text{ }    0.0656  \text{ }    0.0516  \text{ }    0.0803  \text{ }    0.0624  \text{ }    0.0820  \text{ }    0.0537  \text{ }    0.0559  \text{ }    0.0771  \text{ }    0.0108  \text{ }    0.0283  \text{ }    0.1178  \text{ }    0.0403  \text{ }    0.0674  \text{ }    0.0526  \\
0.0829  \text{ }    0.0427  \text{ }    0.0798  \text{ }    0.0389  \text{ }    0.0341  \text{ }    0.0052  \text{ }    0.0162  \text{ }    0.0119  \text{ }    0.0595  \text{ }    0.0566  \text{ }    0.0081  \text{ }    0.0197  \text{ }    0.0525  \text{ }    0.0108  \text{ }    0.1268  \text{ }    0.0481  \text{ }    0.0719  \text{ }    0.1129  \text{ }    0.0020  \text{ }    0.1193  \\
0.0442  \text{ }    0.0088  \text{ }    0.0250  \text{ }    0.0927  \text{ }    0.0006  \text{ }    0.0616  \text{ }    0.0684  \text{ }    0.0535  \text{ }    0.0101  \text{ }    0.0089  \text{ }    0.1193  \text{ }    0.0666  \text{ }    0.0589  \text{ }    0.0283  \text{ }    0.0481  \text{ }    0.0394  \text{ }    0.0645  \text{ }    0.1475  \text{ }    0.0375  \text{ }    0.0161  \\
0.0542  \text{ }    0.0808  \text{ }    0.0757  \text{ }    0.0429  \text{ }    0.0903  \text{ }    0.0347  \text{ }    0.0126  \text{ }    0.0694  \text{ }    0.0311  \text{ }    0.0536  \text{ }    0.0826  \text{ }    0.0065  \text{ }    0.0433  \text{ }    0.1178  \text{ }    0.0719  \text{ }    0.0645  \text{ }    0.0287  \text{ }    0.0104  \text{ }    0.0020  \text{ }    0.0268  \\
0.0112  \text{ }    0.0269  \text{ }    0.0692  \text{ }    0.0272  \text{ }    0.0715  \text{ }    0.0617  \text{ }    0.0582  \text{ }    0.0803  \text{ }    0.0158  \text{ }    0.0441  \text{ }    0.0816  \text{ }    0.0160  \text{ }    0.0657  \text{ }    0.0403  \text{ }    0.1129  \text{ }    0.1475  \text{ }    0.0104  \text{ }    0.0169  \text{ }    0.0344  \text{ }    0.0082  \\
0.0439  \text{ }    0.0845  \text{ }    0.0115  \text{ }    0.0365  \text{ }    0.0263  \text{ }    0.0110  \text{ }    0.0690  \text{ }    0.0165  \text{ }    0.0845  \text{ }    0.0050  \text{ }    0.0243  \text{ }    0.0602  \text{ }    0.0115  \text{ }    0.0674  \text{ }    0.0020  \text{ }    0.0375  \text{ }    0.0020  \text{ }    0.0344  \text{ }    0.1773  \text{ }    0.1946  \\
0.0467  \text{ }    0.0529  \text{ }    0.0151  \text{ }    0.0154
\text{ }    0.0113  \text{ }    0.0288  \text{ }    0.0139  \text{ }
0.0185  \text{ }    0.0597  \text{ }    0.0073  \text{ }    0.0311
\text{ }    0.0455  \text{ }    0.0435  \text{ }    0.0526  \text{ }
0.1193  \text{ }    0.0161  \text{ }    0.0268  \text{ }    0.0082
\text{ }    0.1946  \text{ }    0.1927
\end{matrix} \right] }
\end{eqnarray*}

The transition matrix for arm $3$ is
\begin{eqnarray*} {\tiny \left[  \begin{matrix}
\text{ }    0.0512  \text{ }    0.0398  \text{ }    0.0186  \text{ }    0.1012  \text{ }    0.0880  \text{ }    0.0099  \text{ }    0.0250  \text{ }    0.0948  \text{ }    0.0468  \text{ }    0.0743  \text{ }    0.0810  \text{ }    0.0130  \text{ }    0.0508  \text{ }    0.0513  \text{ }    0.0932  \text{ }    0.0251  \text{ }    0.0051  \text{ }    0.0453  \text{ }    0.0071  \\
\text{ }    0.0122  \text{ }    0.0250  \text{ }    0.0019  \text{ }    0.0156  \text{ }    0.0832  \text{ }    0.0578  \text{ }    0.0962  \text{ }    0.0830  \text{ }    0.0360  \text{ }    0.0522  \text{ }    0.0748  \text{ }    0.0140  \text{ }    0.0884  \text{ }    0.1053  \text{ }    0.0440  \text{ }    0.0094  \text{ }    0.0683  \text{ }    0.0078  \text{ }    0.0737  \\
\text{ }    0.0250  \text{ }    0.0505  \text{ }    0.0590  \text{ }    0.0413  \text{ }    0.0409  \text{ }    0.0907  \text{ }    0.0745  \text{ }    0.0572  \text{ }    0.0256  \text{ }    0.0575  \text{ }    0.0321  \text{ }    0.0725  \text{ }    0.0151  \text{ }    0.0472  \text{ }    0.0051  \text{ }    0.0558  \text{ }    0.0908  \text{ }    0.0418  \text{ }    0.0777  \\
\text{ }    0.0019  \text{ }    0.0590  \text{ }    0.0678  \text{ }    0.0932  \text{ }    0.0800  \text{ }    0.0179  \text{ }    0.0399  \text{ }    0.0393  \text{ }    0.0156  \text{ }    0.0990  \text{ }    0.0134  \text{ }    0.0293  \text{ }    0.0482  \text{ }    0.0440  \text{ }    0.0948  \text{ }    0.0406  \text{ }    0.0155  \text{ }    0.0985  \text{ }    0.0833  \\
\text{ }    0.0156  \text{ }    0.0413  \text{ }    0.0932  \text{ }    0.0061  \text{ }    0.0507  \text{ }    0.0295  \text{ }    0.0588  \text{ }    0.0905  \text{ }    0.0783  \text{ }    0.0488  \text{ }    0.0275  \text{ }    0.0807  \text{ }    0.0171  \text{ }    0.0266  \text{ }    0.0669  \text{ }    0.0320  \text{ }    0.0771  \text{ }    0.0299  \text{ }    0.0281  \\
\text{ }    0.0832  \text{ }    0.0409  \text{ }    0.0800  \text{ }    0.0507  \text{ }    0.0064  \text{ }    0.0550  \text{ }    0.0050  \text{ }    0.0462  \text{ }    0.0698  \text{ }    0.0204  \text{ }    0.0588  \text{ }    0.0313  \text{ }    0.0223  \text{ }    0.0766  \text{ }    0.0465  \text{ }    0.0512  \text{ }    0.0323  \text{ }    0.0767  \text{ }    0.0585  \\
\text{ }    0.0578  \text{ }    0.0907  \text{ }    0.0179  \text{ }    0.0295  \text{ }    0.0550  \text{ }    0.0249  \text{ }    0.0120  \text{ }    0.0223  \text{ }    0.1130  \text{ }    0.0087  \text{ }    0.1138  \text{ }    0.0928  \text{ }    0.0281  \text{ }    0.0060  \text{ }    0.0811  \text{ }    0.0286  \text{ }    0.0814  \text{ }    0.0505  \text{ }    0.0761  \\
\text{ }    0.0962  \text{ }    0.0745  \text{ }    0.0399  \text{ }    0.0588  \text{ }    0.0050  \text{ }    0.0120  \text{ }    0.0709  \text{ }    0.0565  \text{ }    0.0784  \text{ }    0.1050  \text{ }    0.0158  \text{ }    0.0043  \text{ }    0.0885  \text{ }    0.0108  \text{ }    0.0278  \text{ }    0.0529  \text{ }    0.0285  \text{ }    0.1036  \text{ }    0.0455  \\
\text{ }    0.0830  \text{ }    0.0572  \text{ }    0.0393  \text{ }    0.0905  \text{ }    0.0462  \text{ }    0.0223  \text{ }    0.0565  \text{ }    0.0159  \text{ }    0.0607  \text{ }    0.0563  \text{ }    0.0359  \text{ }    0.0259  \text{ }    0.0479  \text{ }    0.0295  \text{ }    0.0455  \text{ }    0.0584  \text{ }    0.0515  \text{ }    0.0379  \text{ }    0.0447  \\
\text{ }    0.0360  \text{ }    0.0256  \text{ }    0.0156  \text{ }    0.0783  \text{ }    0.0698  \text{ }    0.1130  \text{ }    0.0784  \text{ }    0.0607  \text{ }    0.0656  \text{ }    0.0453  \text{ }    0.0347  \text{ }    0.0105  \text{ }    0.0639  \text{ }    0.0236  \text{ }    0.0475  \text{ }    0.0415  \text{ }    0.0695  \text{ }    0.0499  \text{ }    0.0239  \\
\text{ }    0.0522  \text{ }    0.0575  \text{ }    0.0990  \text{ }    0.0488  \text{ }    0.0204  \text{ }    0.0087  \text{ }    0.1050  \text{ }    0.0563  \text{ }    0.0453  \text{ }    0.0744  \text{ }    0.0559  \text{ }    0.0487  \text{ }    0.0183  \text{ }    0.0243  \text{ }    0.0067  \text{ }    0.0421  \text{ }    0.0165  \text{ }    0.0734  \text{ }    0.0723  \\
\text{ }    0.0748  \text{ }    0.0321  \text{ }    0.0134  \text{ }    0.0275  \text{ }    0.0588  \text{ }    0.1138  \text{ }    0.0158  \text{ }    0.0359  \text{ }    0.0347  \text{ }    0.0559  \text{ }    0.0178  \text{ }    0.0155  \text{ }    0.1284  \text{ }    0.0064  \text{ }    0.0132  \text{ }    0.1452  \text{ }    0.0009  \text{ }    0.1132  \text{ }    0.0157  \\
\text{ }    0.0140  \text{ }    0.0725  \text{ }    0.0293  \text{ }    0.0807  \text{ }    0.0313  \text{ }    0.0928  \text{ }    0.0043  \text{ }    0.0259  \text{ }    0.0105  \text{ }    0.0487  \text{ }    0.0155  \text{ }    0.0740  \text{ }    0.1503  \text{ }    0.0878  \text{ }    0.0476  \text{ }    0.0127  \text{ }    0.0679  \text{ }    0.0695  \text{ }    0.0517  \\
\text{ }    0.0884  \text{ }    0.0151  \text{ }    0.0482  \text{ }    0.0171  \text{ }    0.0223  \text{ }    0.0281  \text{ }    0.0885  \text{ }    0.0479  \text{ }    0.0639  \text{ }    0.0183  \text{ }    0.1284  \text{ }    0.1503  \text{ }    0.0253  \text{ }    0.0078  \text{ }    0.0577  \text{ }    0.0046  \text{ }    0.0452  \text{ }    0.0366  \text{ }    0.0554  \\
\text{ }    0.1053  \text{ }    0.0472  \text{ }    0.0440  \text{ }    0.0266  \text{ }    0.0766  \text{ }    0.0060  \text{ }    0.0108  \text{ }    0.0295  \text{ }    0.0236  \text{ }    0.0243  \text{ }    0.0064  \text{ }    0.0878  \text{ }    0.0078  \text{ }    0.0579  \text{ }    0.0753  \text{ }    0.0707  \text{ }    0.1428  \text{ }    0.0106  \text{ }    0.0957  \\
\text{ }    0.0440  \text{ }    0.0051  \text{ }    0.0948  \text{ }    0.0669  \text{ }    0.0465  \text{ }    0.0811  \text{ }    0.0278  \text{ }    0.0455  \text{ }    0.0475  \text{ }    0.0067  \text{ }    0.0132  \text{ }    0.0476  \text{ }    0.0577  \text{ }    0.0753  \text{ }    0.0758  \text{ }    0.0487  \text{ }    0.0863  \text{ }    0.0113  \text{ }    0.0250  \\
\text{ }    0.0094  \text{ }    0.0558  \text{ }    0.0406  \text{ }    0.0320  \text{ }    0.0512  \text{ }    0.0286  \text{ }    0.0529  \text{ }    0.0584  \text{ }    0.0415  \text{ }    0.0421  \text{ }    0.1452  \text{ }    0.0127  \text{ }    0.0046  \text{ }    0.0707  \text{ }    0.0487  \text{ }    0.1239  \text{ }    0.0639  \text{ }    0.0421  \text{ }    0.0507  \\
\text{ }    0.0683  \text{ }    0.0908  \text{ }    0.0155  \text{ }    0.0771  \text{ }    0.0323  \text{ }    0.0814  \text{ }    0.0285  \text{ }    0.0515  \text{ }    0.0695  \text{ }    0.0165  \text{ }    0.0009  \text{ }    0.0679  \text{ }    0.0452  \text{ }    0.1428  \text{ }    0.0863  \text{ }    0.0639  \text{ }    0.0215  \text{ }    0.0200  \text{ }    0.0150  \\
\text{ }    0.0078  \text{ }    0.0418  \text{ }    0.0985  \text{ }    0.0299  \text{ }    0.0767  \text{ }    0.0505  \text{ }    0.1036  \text{ }    0.0379  \text{ }    0.0499  \text{ }    0.0734  \text{ }    0.1132  \text{ }    0.0695  \text{ }    0.0366  \text{ }    0.0106  \text{ }    0.0113  \text{ }    0.0421  \text{ }    0.0200  \text{ }    0.0228  \text{ }    0.0586  \\
\text{ }    0.0737  \text{ }    0.0777  \text{ }    0.0833  \text{ }
0.0281  \text{ }    0.0585  \text{ }    0.0761  \text{ }    0.0455
\text{ }    0.0447  \text{ }    0.0239  \text{ }    0.0723  \text{ }
0.0157  \text{ }    0.0517  \text{ }    0.0554  \text{ }    0.0957
\text{ }    0.0250  \text{ }    0.0507  \text{ }    0.0150  \text{ }
0.0586  \text{ }    0.0415
\end{matrix} \right] }
\end{eqnarray*}

The transition matrix for arm $4$ is
\begin{eqnarray*} {\tiny \left[  \begin{matrix}0.0541  \text{ }    0.1087  \text{ }    0.0564  \text{ }    0.0311  \text{ }    0.0663  \text{ }    0.0134  \text{ }    0.0580  \text{ }    0.0345  \text{ }    0.0239  \text{ }    0.0808  \text{ }    0.0066  \text{ }    0.0038  \text{ }    0.0628  \text{ }    0.1039  \text{ }    0.0002  \text{ }    0.0138  \text{ }    0.0820  \text{ }    0.0606  \text{ }    0.1037  \text{ }    0.0354  \\
0.1087  \text{ }    0.0879  \text{ }    0.0148  \text{ }    0.0524  \text{ }    0.0073  \text{ }    0.0300  \text{ }    0.0403  \text{ }    0.0335  \text{ }    0.0768  \text{ }    0.0253  \text{ }    0.0651  \text{ }    0.0916  \text{ }    0.0188  \text{ }    0.0305  \text{ }    0.0473  \text{ }    0.0201  \text{ }    0.0574  \text{ }    0.0742  \text{ }    0.0728  \text{ }    0.0453  \\
0.0564  \text{ }    0.0148  \text{ }    0.0073  \text{ }    0.0739  \text{ }    0.0583  \text{ }    0.1106  \text{ }    0.0338  \text{ }    0.0907  \text{ }    0.0350  \text{ }    0.1033  \text{ }    0.0850  \text{ }    0.0352  \text{ }    0.0502  \text{ }    0.0254  \text{ }    0.0001  \text{ }    0.0534  \text{ }    0.0035  \text{ }    0.0307  \text{ }    0.0543  \text{ }    0.0782  \\
0.0311  \text{ }    0.0524  \text{ }    0.0739  \text{ }    0.0546  \text{ }    0.0103  \text{ }    0.0771  \text{ }    0.0710  \text{ }    0.0422  \text{ }    0.0591  \text{ }    0.0650  \text{ }    0.0709  \text{ }    0.0875  \text{ }    0.0371  \text{ }    0.0019  \text{ }    0.0627  \text{ }    0.0593  \text{ }    0.0338  \text{ }    0.0454  \text{ }    0.0287  \text{ }    0.0362  \\
0.0663  \text{ }    0.0073  \text{ }    0.0583  \text{ }    0.0103  \text{ }    0.0418  \text{ }    0.0822  \text{ }    0.0536  \text{ }    0.0701  \text{ }    0.0631  \text{ }    0.0506  \text{ }    0.0820  \text{ }    0.0090  \text{ }    0.0219  \text{ }    0.0802  \text{ }    0.0503  \text{ }    0.0813  \text{ }    0.0045  \text{ }    0.0793  \text{ }    0.0475  \text{ }    0.0406  \\
0.0134  \text{ }    0.0300  \text{ }    0.1106  \text{ }    0.0771  \text{ }    0.0822  \text{ }    0.0139  \text{ }    0.0589  \text{ }    0.0044  \text{ }    0.0686  \text{ }    0.0496  \text{ }    0.0350  \text{ }    0.0764  \text{ }    0.0585  \text{ }    0.0368  \text{ }    0.0525  \text{ }    0.0445  \text{ }    0.0894  \text{ }    0.0301  \text{ }    0.0270  \text{ }    0.0410  \\
0.0580  \text{ }    0.0403  \text{ }    0.0338  \text{ }    0.0710  \text{ }    0.0536  \text{ }    0.0589  \text{ }    0.0145  \text{ }    0.0642  \text{ }    0.0416  \text{ }    0.0223  \text{ }    0.0800  \text{ }    0.0787  \text{ }    0.0549  \text{ }    0.0090  \text{ }    0.0782  \text{ }    0.0142  \text{ }    0.0765  \text{ }    0.0188  \text{ }    0.0697  \text{ }    0.0619  \\
0.0345  \text{ }    0.0335  \text{ }    0.0907  \text{ }    0.0422  \text{ }    0.0701  \text{ }    0.0044  \text{ }    0.0642  \text{ }    0.0634  \text{ }    0.0144  \text{ }    0.0200  \text{ }    0.0927  \text{ }    0.0180  \text{ }    0.0009  \text{ }    0.0644  \text{ }    0.0832  \text{ }    0.0987  \text{ }    0.0809  \text{ }    0.0629  \text{ }    0.0479  \text{ }    0.0130  \\
0.0239  \text{ }    0.0768  \text{ }    0.0350  \text{ }    0.0591  \text{ }    0.0631  \text{ }    0.0686  \text{ }    0.0416  \text{ }    0.0144  \text{ }    0.0169  \text{ }    0.0414  \text{ }    0.0515  \text{ }    0.0010  \text{ }    0.1144  \text{ }    0.0228  \text{ }    0.0261  \text{ }    0.0652  \text{ }    0.0913  \text{ }    0.0643  \text{ }    0.0326  \text{ }    0.0902  \\
0.0808  \text{ }    0.0253  \text{ }    0.1033  \text{ }    0.0650  \text{ }    0.0506  \text{ }    0.0496  \text{ }    0.0223  \text{ }    0.0200  \text{ }    0.0414  \text{ }    0.0721  \text{ }    0.0788  \text{ }    0.0083  \text{ }    0.0653  \text{ }    0.0118  \text{ }    0.0399  \text{ }    0.0256  \text{ }    0.0774  \text{ }    0.0260  \text{ }    0.0669  \text{ }    0.0693  \\
0.0066  \text{ }    0.0651  \text{ }    0.0850  \text{ }    0.0709  \text{ }    0.0820  \text{ }    0.0350  \text{ }    0.0800  \text{ }    0.0927  \text{ }    0.0515  \text{ }    0.0788  \text{ }    0.0467  \text{ }    0.0323  \text{ }    0.0446  \text{ }    0.0507  \text{ }    0.0113  \text{ }    0.0668  \text{ }    0.0021  \text{ }    0.0077  \text{ }    0.0315  \text{ }    0.0588  \\
0.0038  \text{ }    0.0916  \text{ }    0.0352  \text{ }    0.0875  \text{ }    0.0090  \text{ }    0.0764  \text{ }    0.0787  \text{ }    0.0180  \text{ }    0.0010  \text{ }    0.0083  \text{ }    0.0323  \text{ }    0.0501  \text{ }    0.1065  \text{ }    0.0553  \text{ }    0.0688  \text{ }    0.0594  \text{ }    0.1087  \text{ }    0.0678  \text{ }    0.0239  \text{ }    0.0176  \\
0.0628  \text{ }    0.0188  \text{ }    0.0502  \text{ }    0.0371  \text{ }    0.0219  \text{ }    0.0585  \text{ }    0.0549  \text{ }    0.0009  \text{ }    0.1144  \text{ }    0.0653  \text{ }    0.0446  \text{ }    0.1065  \text{ }    0.0716  \text{ }    0.0544  \text{ }    0.0577  \text{ }    0.0528  \text{ }    0.0564  \text{ }    0.0314  \text{ }    0.0264  \text{ }    0.0135  \\
0.1039  \text{ }    0.0305  \text{ }    0.0254  \text{ }    0.0019  \text{ }    0.0802  \text{ }    0.0368  \text{ }    0.0090  \text{ }    0.0644  \text{ }    0.0228  \text{ }    0.0118  \text{ }    0.0507  \text{ }    0.0553  \text{ }    0.0544  \text{ }    0.0208  \text{ }    0.0810  \text{ }    0.0757  \text{ }    0.0430  \text{ }    0.0629  \text{ }    0.0666  \text{ }    0.1029  \\
0.0002  \text{ }    0.0473  \text{ }    0.0001  \text{ }    0.0627  \text{ }    0.0503  \text{ }    0.0525  \text{ }    0.0782  \text{ }    0.0832  \text{ }    0.0261  \text{ }    0.0399  \text{ }    0.0113  \text{ }    0.0688  \text{ }    0.0577  \text{ }    0.0810  \text{ }    0.0602  \text{ }    0.0683  \text{ }    0.0558  \text{ }    0.0616  \text{ }    0.0665  \text{ }    0.0283  \\
0.0138  \text{ }    0.0201  \text{ }    0.0534  \text{ }    0.0593  \text{ }    0.0813  \text{ }    0.0445  \text{ }    0.0142  \text{ }    0.0987  \text{ }    0.0652  \text{ }    0.0256  \text{ }    0.0668  \text{ }    0.0594  \text{ }    0.0528  \text{ }    0.0757  \text{ }    0.0683  \text{ }    0.0637  \text{ }    0.0400  \text{ }    0.0132  \text{ }    0.0436  \text{ }    0.0406  \\
0.0820  \text{ }    0.0574  \text{ }    0.0035  \text{ }    0.0338  \text{ }    0.0045  \text{ }    0.0894  \text{ }    0.0765  \text{ }    0.0809  \text{ }    0.0913  \text{ }    0.0774  \text{ }    0.0021  \text{ }    0.1087  \text{ }    0.0564  \text{ }    0.0430  \text{ }    0.0558  \text{ }    0.0400  \text{ }    0.0584  \text{ }    0.0077  \text{ }    0.0127  \text{ }    0.0183  \\
0.0606  \text{ }    0.0742  \text{ }    0.0307  \text{ }    0.0454  \text{ }    0.0793  \text{ }    0.0301  \text{ }    0.0188  \text{ }    0.0629  \text{ }    0.0643  \text{ }    0.0260  \text{ }    0.0077  \text{ }    0.0678  \text{ }    0.0314  \text{ }    0.0629  \text{ }    0.0616  \text{ }    0.0132  \text{ }    0.0077  \text{ }    0.0011  \text{ }    0.1112  \text{ }    0.1433  \\
0.1037  \text{ }    0.0728  \text{ }    0.0543  \text{ }    0.0287  \text{ }    0.0475  \text{ }    0.0270  \text{ }    0.0697  \text{ }    0.0479  \text{ }    0.0326  \text{ }    0.0669  \text{ }    0.0315  \text{ }    0.0239  \text{ }    0.0264  \text{ }    0.0666  \text{ }    0.0665  \text{ }    0.0436  \text{ }    0.0127  \text{ }    0.1112  \text{ }    0.0199  \text{ }    0.0468  \\
0.0354  \text{ }    0.0453  \text{ }    0.0782  \text{ }    0.0362
\text{ }    0.0406  \text{ }    0.0410  \text{ }    0.0619  \text{ }
0.0130  \text{ }    0.0902  \text{ }    0.0693  \text{ }    0.0588
\text{ }    0.0176  \text{ }    0.0135  \text{ }    0.1029  \text{ }
0.0283  \text{ }    0.0406  \text{ }    0.0183  \text{ }    0.1433
\text{ }    0.0468  \text{ }    0.0189
\end{matrix} \right] }
\end{eqnarray*}

The transition matrix for arm $5$ is
\begin{eqnarray*} {\tiny \left[  \begin{matrix}0.0611  \text{ }    0.0525  \text{ }    0.0272  \text{ }    0.0718  \text{ }    0.0593  \text{ }    0.0697  \text{ }    0.0023  \text{ }    0.0739  \text{ }    0.0604  \text{ }    0.0050  \text{ }    0.0332  \text{ }    0.0591  \text{ }    0.0672  \text{ }    0.0666  \text{ }    0.0356  \text{ }    0.0759  \text{ }    0.0547  \text{ }    0.0325  \text{ }    0.0372  \text{ }    0.0547  \\
0.0525  \text{ }    0.0161  \text{ }    0.0462  \text{ }    0.0541  \text{ }    0.1034  \text{ }    0.0507  \text{ }    0.0334  \text{ }    0.0598  \text{ }    0.0166  \text{ }    0.0898  \text{ }    0.0464  \text{ }    0.0355  \text{ }    0.0222  \text{ }    0.0418  \text{ }    0.0835  \text{ }    0.1064  \text{ }    0.0105  \text{ }    0.0438  \text{ }    0.0657  \text{ }    0.0217  \\
0.0272  \text{ }    0.0462  \text{ }    0.1379  \text{ }    0.0795  \text{ }    0.0635  \text{ }    0.0036  \text{ }    0.0129  \text{ }    0.0787  \text{ }    0.0591  \text{ }    0.1029  \text{ }    0.0232  \text{ }    0.0384  \text{ }    0.0227  \text{ }    0.0431  \text{ }    0.0048  \text{ }    0.0792  \text{ }    0.0704  \text{ }    0.0131  \text{ }    0.0315  \text{ }    0.0622  \\
0.0718  \text{ }    0.0541  \text{ }    0.0795  \text{ }    0.0639  \text{ }    0.0740  \text{ }    0.0367  \text{ }    0.0826  \text{ }    0.0601  \text{ }    0.0610  \text{ }    0.0801  \text{ }    0.0057  \text{ }    0.0070  \text{ }    0.0243  \text{ }    0.0363  \text{ }    0.0281  \text{ }    0.0499  \text{ }    0.0360  \text{ }    0.0599  \text{ }    0.0498  \text{ }    0.0393  \\
0.0593  \text{ }    0.1034  \text{ }    0.0635  \text{ }    0.0740  \text{ }    0.0584  \text{ }    0.0398  \text{ }    0.0586  \text{ }    0.0628  \text{ }    0.0672  \text{ }    0.0195  \text{ }    0.0422  \text{ }    0.0313  \text{ }    0.0263  \text{ }    0.0353  \text{ }    0.0151  \text{ }    0.0555  \text{ }    0.0477  \text{ }    0.0525  \text{ }    0.0362  \text{ }    0.0514  \\
0.0697  \text{ }    0.0507  \text{ }    0.0036  \text{ }    0.0367  \text{ }    0.0398  \text{ }    0.0149  \text{ }    0.0850  \text{ }    0.0824  \text{ }    0.0158  \text{ }    0.0456  \text{ }    0.0818  \text{ }    0.0078  \text{ }    0.0058  \text{ }    0.0613  \text{ }    0.0803  \text{ }    0.0978  \text{ }    0.0253  \text{ }    0.0882  \text{ }    0.0982  \text{ }    0.0092  \\
0.0023  \text{ }    0.0334  \text{ }    0.0129  \text{ }    0.0826  \text{ }    0.0586  \text{ }    0.0850  \text{ }    0.1057  \text{ }    0.0178  \text{ }    0.1326  \text{ }    0.0105  \text{ }    0.0110  \text{ }    0.0419  \text{ }    0.0443  \text{ }    0.0902  \text{ }    0.0344  \text{ }    0.0269  \text{ }    0.0267  \text{ }    0.0364  \text{ }    0.1264  \text{ }    0.0205  \\
0.0739  \text{ }    0.0598  \text{ }    0.0787  \text{ }    0.0601  \text{ }    0.0628  \text{ }    0.0824  \text{ }    0.0178  \text{ }    0.0526  \text{ }    0.0528  \text{ }    0.0043  \text{ }    0.0139  \text{ }    0.0514  \text{ }    0.0190  \text{ }    0.0710  \text{ }    0.0862  \text{ }    0.0571  \text{ }    0.0827  \text{ }    0.0215  \text{ }    0.0489  \text{ }    0.0029  \\
0.0604  \text{ }    0.0166  \text{ }    0.0591  \text{ }    0.0610  \text{ }    0.0672  \text{ }    0.0158  \text{ }    0.1326  \text{ }    0.0528  \text{ }    0.0065  \text{ }    0.1016  \text{ }    0.0898  \text{ }    0.0422  \text{ }    0.0708  \text{ }    0.0175  \text{ }    0.0355  \text{ }    0.0013  \text{ }    0.0007  \text{ }    0.0828  \text{ }    0.0492  \text{ }    0.0366  \\
0.0050  \text{ }    0.0898  \text{ }    0.1029  \text{ }    0.0801  \text{ }    0.0195  \text{ }    0.0456  \text{ }    0.0105  \text{ }    0.0043  \text{ }    0.1016  \text{ }    0.0714  \text{ }    0.0483  \text{ }    0.0345  \text{ }    0.0399  \text{ }    0.0451  \text{ }    0.0783  \text{ }    0.0747  \text{ }    0.0488  \text{ }    0.0681  \text{ }    0.0269  \text{ }    0.0046  \\
0.0332  \text{ }    0.0464  \text{ }    0.0232  \text{ }    0.0057  \text{ }    0.0422  \text{ }    0.0818  \text{ }    0.0110  \text{ }    0.0139  \text{ }    0.0898  \text{ }    0.0483  \text{ }    0.1044  \text{ }    0.0361  \text{ }    0.1064  \text{ }    0.0848  \text{ }    0.0101  \text{ }    0.0608  \text{ }    0.0084  \text{ }    0.0761  \text{ }    0.0740  \text{ }    0.0434  \\
0.0591  \text{ }    0.0355  \text{ }    0.0384  \text{ }    0.0070  \text{ }    0.0313  \text{ }    0.0078  \text{ }    0.0419  \text{ }    0.0514  \text{ }    0.0422  \text{ }    0.0345  \text{ }    0.0361  \text{ }    0.0196  \text{ }    0.1013  \text{ }    0.0313  \text{ }    0.0383  \text{ }    0.0796  \text{ }    0.0969  \text{ }    0.0829  \text{ }    0.0744  \text{ }    0.0904  \\
0.0672  \text{ }    0.0222  \text{ }    0.0227  \text{ }    0.0243  \text{ }    0.0263  \text{ }    0.0058  \text{ }    0.0443  \text{ }    0.0190  \text{ }    0.0708  \text{ }    0.0399  \text{ }    0.1064  \text{ }    0.1013  \text{ }    0.0824  \text{ }    0.0914  \text{ }    0.0382  \text{ }    0.0861  \text{ }    0.0339  \text{ }    0.0405  \text{ }    0.0426  \text{ }    0.0345  \\
0.0666  \text{ }    0.0418  \text{ }    0.0431  \text{ }    0.0363  \text{ }    0.0353  \text{ }    0.0613  \text{ }    0.0902  \text{ }    0.0710  \text{ }    0.0175  \text{ }    0.0451  \text{ }    0.0848  \text{ }    0.0313  \text{ }    0.0914  \text{ }    0.0154  \text{ }    0.0278  \text{ }    0.0378  \text{ }    0.0019  \text{ }    0.0942  \text{ }    0.0452  \text{ }    0.0620  \\
0.0356  \text{ }    0.0835  \text{ }    0.0048  \text{ }    0.0281  \text{ }    0.0151  \text{ }    0.0803  \text{ }    0.0344  \text{ }    0.0862  \text{ }    0.0355  \text{ }    0.0783  \text{ }    0.0101  \text{ }    0.0383  \text{ }    0.0382  \text{ }    0.0278  \text{ }    0.0463  \text{ }    0.1093  \text{ }    0.1148  \text{ }    0.0716  \text{ }    0.0386  \text{ }    0.0233  \\
0.0759  \text{ }    0.1064  \text{ }    0.0792  \text{ }    0.0499  \text{ }    0.0555  \text{ }    0.0978  \text{ }    0.0269  \text{ }    0.0571  \text{ }    0.0013  \text{ }    0.0747  \text{ }    0.0608  \text{ }    0.0796  \text{ }    0.0861  \text{ }    0.0378  \text{ }    0.1093  \text{ }    0.0001  \text{ }    0.0002  \text{ }    0.0006  \text{ }    0.0005  \text{ }    0.0002  \\
0.0547  \text{ }    0.0105  \text{ }    0.0704  \text{ }    0.0360  \text{ }    0.0477  \text{ }    0.0253  \text{ }    0.0267  \text{ }    0.0827  \text{ }    0.0007  \text{ }    0.0488  \text{ }    0.0084  \text{ }    0.0969  \text{ }    0.0339  \text{ }    0.0019  \text{ }    0.1148  \text{ }    0.0002  \text{ }    0.2200  \text{ }    0.0147  \text{ }    0.0939  \text{ }    0.0119  \\
0.0325  \text{ }    0.0438  \text{ }    0.0131  \text{ }    0.0599  \text{ }    0.0525  \text{ }    0.0882  \text{ }    0.0364  \text{ }    0.0215  \text{ }    0.0828  \text{ }    0.0681  \text{ }    0.0761  \text{ }    0.0829  \text{ }    0.0405  \text{ }    0.0942  \text{ }    0.0716  \text{ }    0.0006  \text{ }    0.0147  \text{ }    0.0070  \text{ }    0.0366  \text{ }    0.0769  \\
0.0372  \text{ }    0.0657  \text{ }    0.0315  \text{ }    0.0498  \text{ }    0.0362  \text{ }    0.0982  \text{ }    0.1264  \text{ }    0.0489  \text{ }    0.0492  \text{ }    0.0269  \text{ }    0.0740  \text{ }    0.0744  \text{ }    0.0426  \text{ }    0.0452  \text{ }    0.0386  \text{ }    0.0005  \text{ }    0.0939  \text{ }    0.0366  \text{ }    0.0034  \text{ }    0.0207  \\
0.0547  \text{ }    0.0217  \text{ }    0.0622  \text{ }    0.0393
\text{ }    0.0514  \text{ }    0.0092  \text{ }    0.0205  \text{ }
0.0029  \text{ }    0.0366  \text{ }    0.0046  \text{ }    0.0434
\text{ }    0.0904  \text{ }    0.0345  \text{ }    0.0620  \text{ }
0.0233  \text{ }    0.0002  \text{ }    0.0119  \text{ }    0.0769
\text{ }    0.0207  \text{ }    0.3338
\end{matrix} \right] }
\end{eqnarray*}


\begin{thebibliography}{1}






\bibitem{Lai&Robbins85AAM}
T. Lai and H. Robbins, ``Asymptotically Efficient Adaptive
Allocation Rules,'' {\em Advances in Applied Mathematics}, vol.~6,
no.~1, pp.~4¨C22, 1985.


\bibitem{Agrawal:95}
R. Agrawal, ``Sample Mean Based Index Policies With O(log n) Regret
for the Multi-armed Bandit Problem,'' {\em Advances in Applied
Probability}, vol.~27, pp.~1054¨C1078, 1995.

\bibitem{Auer:02}
P. Auer, N. Cesa-Bianchi, P. Fischer, ``Finite-time Analysis of the
Multiarmed Bandit Problem,'' {\em Machine Learning}, 47, 235-256,
2002.


\bibitem{qingzhaosub}
K. Liu and Q. Zhao "Deterministic Sequencing of Exploration and
Exploitation for Multi-Armed Bandit Problems," {\em Proc. of
Allerton Conference on Communications, Control, and Computing},
Sep., 2011.

\bibitem{Anantharam:87-2}
V. Anantharam, P. Varaiya, J. Walrand, ``Asymptotically Efficient
Allocation Rules for the Multiarmed Bandit Problem with Multiple
Plays-Part II: Markovian Rewards,'' {\em IEEE Transaction on
Automatic Control}, vol.~AC-32, no.11, pp.~977-982, Nov., 1987.
\bibitem{Tekin:10}
C. Tekin, M. Liu, ``Online Algorithms for the Multi-Armed Bandit
Problem With Markovian Rewards,'' {\em Proc. of Allerton Conference
on Communications, Control, and Computing}, Sep., 2010.

\bibitem{Papadimitriou:99}
C. Papadimitriou, J. Tsitsiklis, ``The Complexity of Optimal Queuing
Network Control,'' {\em Mathematics of Operations Research}, vol.
24, no. 2, pp. 293-305, May, 1999.



\bibitem{Auer-nonsto}
P.~Auer, N.~Cesa-Bianchi, Y.~Freund, R.E.~Schapire ``The
nonstochastic multiarmed bandit problem,'' {\em SIAM Journal on
Computing}, vol. 32, pp. 48¨C77, 2002.




\bibitem{Tekin:10-2}
C.~Tekin, M.~Liu, ``Online Learning in Opportunistic Spectrum
Access: A Restless Bandit Approach,'' {\em Proc. of Internanional
Conference on  Computer Communications (INFOCOM)}, April 2011,
Shanghai, China.


\bibitem{QZ10}
W.~Dai, Y.~Gai, B.~ Krishnamachari, Q.~Zhao ``The Non-Bayesian
Restless Multi-armed Bandit: A Case Of Near-Logarithmic Regret,'' {\em Proc. of Internanional Conference on Acoustics, Speech and Signal Processing (ICASSP)}, May, 2011. 



\bibitem{QZTWC}
Q.~Zhao, B.~Krishnamachari, and K.~Liu, ``On Myopic Sensing for
Multi-Channel Opportunistic Access: Structure, Optimality, and
Performance,'' {\em IEEE Transactions on Wireless Communications},
vol.~7, no.~12, pp.~5431-5440, Dec., 2008.

\bibitem{CSZ}
S.H.~Ahmad, M.~Liu, T.~Javidi, Q.~Zhao, B.~Krishnamachari
``Optimality of Myopic Sensing in Multi-Channel Opportunistic
Access,'' {\em
 IEEE Transactions on Information Theory}, vol.~55, No.~9, pp.~4040-4050, Sep., 2009.



\bibitem{Keqin:10}
K. Liu, Q. Zhao, ``Distributed Learning in Multi-Armed Bandit with
Multiple Players,'' {\em IEEE Transations on Signal Processing},
vol.~58, no.~11, pp.~5667-5681, Nov., 2010.





\bibitem{Anima}
A. Anandkumar, N. Michael, A.K. Tang, A. Swami ``Distributed
Algorithms for Learning and Cognitive Medium Access with Logarithmic
Regret,'' {\em IEEE JSAC on Advances in Cognitive Radio Networking
and Communications}, vol.~29, no.~4, pp.~731-745, Mar., 2011.

\bibitem{Gaiyi}
Y. Gai and B. Krishnamachari, ``Decentralized Online Learning
Algorithms for Opportunistic Spectrum Access,''  {\em  IEEE Global
Communications Conference (GLOBECOM 2011)}, Houston, USA, Dec.,
2011.




\bibitem{R1} R. Bellman, ``A Problem in the Sequential Design of
Experiments,'' {\em Sankhia}, vol. 16, pp. 221-229, 1956.

\bibitem{R2} J. Gittins, ``Bandit Processes and Dynamic Allocation
Indices,'' {\em Journal of the Royal Statistical Society}, vol. 41,
no. 2, pp. 148177, 1979.

\bibitem{R3} P. Whittle, ``Restless Bandits: Activity Allocation in a
Changing World,'' {\em J. Appl. Probab.}, vol. 25, pp. 287-298,
1988.

\bibitem{R4} R. R.Weber and G.Weiss, ``On an Index Policy for Restless
Bandits,'' {\em J. Appl. Probab.}, vol.27, no.3, pp. 637-648, Sep.,
1990.
\bibitem{R5} R. R. Weber and G. Weiss,  ``Addendum to ¡¯On an
Index Policy for Restless Bandits,'' {\em Adv. Appl. Prob.}, vol.
23, no. 2, pp. 429-430, Jun., 1991.

\bibitem{R6} K. D. Glazebrook, H. M. Mitchell,  ``An Index Policy for a
Stochastic Scheduling Model with Improving/ Deteriorating Jobs,''
{\em Naval Research Logistics (NRL)}, vol. 49, pp. 706-721, Mar.,
2002.

\bibitem{R7} P. S. Ansell, K. D. Glazebrook, J.E. Nino-Mora, and M.
O¡¯Keeffe, ``Whittle¡¯s Index Policy for a Multi-Class Queueing
System with Convex Holding Costs,'' {\em Math. Meth. Operat. Res.},
vol. 57, pp. 21-39, 2003.


\bibitem{R8} K. D. Glazebrook, D. Ruiz-Hernandez, and C. Kirkbride, ``Some
Indexable Families of Restless Bandit Problems,'' {\em Advances in
Applied Probability}, vol. 38, pp. 643-672, 2006.
\bibitem{R9}
K.~Liu and Q.~Zhao ``Indexability of Restless Bandit Problems and
Optimality of Whittle Index for Dynamic Multichannel Access,'' {\em
IEEE Transactions on Information Theory}, vol.~55, no.~11,
pp.~5547-5567, Nov. 2010.


\bibitem{QZCR}
Q.~Zhao and B.M.~Sadler ``A Survey of Dynamic Spectrum Access,''
{\em IEEE Signal Processing Magazine}, vol.~24, no.~3, pp.~79-89,
May. 2007.




\bibitem{AVS}
M.~Agarwal, V.~S.~Borkar, A.~Karandikar,  ``Structural Properties of
Optimal Transmission Policies Over a Randomly Varying Channel,''
{\em IEEE Transactions on Wireless Communications}, vol.~53, no.~6,
pp.~1476-1491, Jul. 2008.


\bibitem{AVS1}
S.~Ali, V.~Krishnamurthy, V.~Leung, `` Optimal and Approximate
Mobility Assisted Opportunistic Scheduling in Cellular Data
Networks,'' {\em IEEE Transactions Mobile Computing}, vol. 6, no. 6,
pp.~633-648, Jun. 2007

\bibitem{AVS2}
L.~Johnston and Vikram.~Krishnamurthy, ``Opportunistic File Transfer
over a Fading Channel - A POMDP Search Theory Formulation with
Optimal Threshold Policies,'' {\em IEEE Transactions Wireless
Communications}, vol.~5, no.~2, pp.~394-405, Feb. 2006.

\bibitem{VC}
M.~Sorensen ``Learning By Investing: Evidence from Venture
Capital,'' {\em Proc. of American Finance Association Annual
Meeting}, Feb. 2008.



\bibitem{Gillman:98}
D. Gillman, ``A Chernoff Bound for Random Walks on Expander
Graphs,'' {\em Proc. 34th IEEE Symp. on Foundatioins of Computer
Science (FOCS93),vol. SIAM J. Comp.}, vol. 27, no. 4, 1998.




\end{thebibliography}
\end{document}